\documentclass[a4paper,10pt,fleqn]{article}
\usepackage[english]{babel}
\usepackage[fleqn,reqno]{amsmath}
\usepackage{amsfonts}
\usepackage{amstext}
\usepackage{amsbsy}
\usepackage{amsthm}
\usepackage{amssymb}
\usepackage{url}
\usepackage{graphicx}
\usepackage{bbm}
\usepackage{xcolor}
\usepackage{comment}

\newtheorem{theorem}{Th\'eor\`eme}

\newtheorem{definition}[theorem]{Definition}
\newtheorem{remark}[theorem]{Remark}
\newtheorem{proposition}{Proposition}

\newcommand{\bea}{\begin{eqnarray}}
\newcommand{\eea}{\end{eqnarray}}
\newcommand{\beas}{\begin{eqnarray*}}
\newcommand{\eeas}{\end{eqnarray*}}

\def\[{\left[}
\def\]{\right]}
\def\<{\langle}
\def\>{\rangle}
\def\({\left(}
\def\){\right)}
\def\a{\alpha}

\newcommand{\Dt}{\partial_t}

\newcommand{\veps}{\varepsilon}
\newcommand{\Dlt}{\Delta t}

\newcommand{\dvg}{\nabla \cdot }
\newcommand{\mbb}{\mathbb}

\makeatletter
\usepackage{authblk}
\usepackage{lipsum}
\usepackage{algorithm} 
\usepackage{algpseudocode} 
\providecommand{\keywords}[1]
{
  \small	
  \textbf{\textit{Keywords---}} #1
}

\begin{document}

\title{\textbf{High order Asymptotic Preserving penalized numerical schemes for the Euler-Poisson system in the quasi-neutral limit}}
\author[1]{\small \textbf{Nicolas Crouseilles}} 
\author[2]{\small \textbf{Giacomo Dimarco}}
\author[3]{\small \textbf{Saurav Samantaray}} 
\affil[1]{Universit\'e de Rennes, Inria Rennes (Mingus team) and IRMAR UMR CNRS 6625, F-35042 Rennes, France \& ENS Rennes}
\affil[2]{Department of Mathematics and Computer Science \&
Center for Computing, Modeling and Statistics (CMCS), University of Ferrare, Italy}
\affil[3]{Department of Mathematics, Indian Institute Of Technology Madras, Chennai, India}
\date{}

\maketitle

\begin{abstract}
In this work, we focus on the development of high-order Implicit-Explicit (IMEX) finite volume numerical methods for plasmas in quasineutral regimes. At large temporal and spatial scales, plasmas tend to be quasineutral, meaning that the local net charge density is nearly zero. However, at small time and spatial scales, measured by the the Debye length, quasineutrality breaks down. In such regimes, standard numerical methods face severe stability constraints, rendering them practically unusable. To address this issue, we introduce and analyze a class of penalized IMEX Runge-Kutta methods for the Euler-Poisson (EP) system, specifically designed to handle the quasineutral limit. These schemes are uniformly stable with respect to the Debye length and degenerate into high-order methods as the quasineutral limit is approached. Several numerical tests confirm that the proposed methods exhibit the desired properties.
\end{abstract}

\tableofcontents

\section{Introduction}
The modeling and numerical simulation of plasma phenomena is a highly active field of research, driven by its significance in both applied technologies and fundamental scientific exploration. Plasmas, often referred to as the fourth state of matter, are ionized gases that exhibit unique behaviors under the influence of electromagnetic fields \cite{Chen}. Their study is crucial for a broad range of technological applications, such as in the development of fusion energy, where confined plasmas hold the potential for generating nearly limitless clean energy \cite{Dimits2000969}. Plasmas are also essential in industries like semiconductor manufacturing, space propulsion systems, medical devices, and environmental control technologies, where plasma processes are used for applications such as waste treatment and air purification.

In this paper, we address the development of numerical methods for plasmas, focusing specifically on the challenges associated with handling quasineutrality. Our study concentrates on a one-fluid Euler-Poisson (EP) system, which models the dynamics of electrons in a plasma with a constant ion background. In this system, the density and momentum evolve according to a hyperbolic problem influenced by an electrostatic potential that solves an elliptic equation. Despite of its simplicity, it contains most of the challenges related to the treatment of the quasi neutrality in plasmas.

An important physical parameter in this context is the Debye length, $\lambda$, which measures the typical length scale of charge imbalances in a plasma. In quasineutral regions, the electric charges effectively cancel out. However, classical discretization methods fail to capture the small scales corresponding to the Debye length, which can be much smaller than the overall size of the system. Using explicit time discretization schemes in such cases leads to prohibitively expensive simulations \cite{Deg13}.

Furthermore, many practical situations, such as sheath formation, plasma diode modeling, or arc formation on satellite solar cells, involve the coexistence of both quasineutral and non-quasineutral regions. These scenarios require specialized numerical treatments. One promising approach involves the use of domain decomposition techniques, which have proven highly effective in managing multiple regimes in plasmas \cite{Dimarco_rispoli}. These methods offer a viable alternative to standard explicit numerical schemes that would otherwise necessitate resolving the smallest scale (the Debye length), leading to prohibitively high computational costs.
However, to the best of our knowledge, the use of domain decomposition methods to connect the asymptotic quasineutral model with the Euler-Poisson system has not yet been explored. This connection presents unique challenges, as developing appropriate numerical methods and accurately identifying the interface between different models requires specific attention. The task of ensuring a smooth transition between the models is complex and demands further research and development.

Several alternative approaches have been proposed in the literature to address the challenge posed by the coexistence of different regimes in plasmas. Among the most successful in recent years are the so-called Asymptotic-Preserving (AP) schemes \cite{ap_jin, ap_jin2}. These schemes are uniformly stable with respect to small parameters and transition smoothly into a consistent numerical scheme for the asymptotic model as the parameter tends to zero. In other words, the limit of the AP scheme is both valid and accurate for the corresponding limiting model.

For the Euler-Poisson system in the quasineutral limit, notable works include \cite{CDV_ASP_05, CDV07, Deg13, ACS24, ALVAREZLAGUNA2020109634}, as well as \cite{fabre, DLV08}, which provide comprehensive numerical analyses of the problem. Additional relevant studies on numerical methods for handling quasineutrality can be found in \cite{Bessemoulin, Belaouar, Crispel, Degond_deluzet, Crouseilles, Labrunie}.
More generally, Asymptotic-Preserving schemes have garnered significant interest in recent decades for their ability to approximate multiscale partial differential equations (PDEs), which are ubiquitous in plasma physics. For further exploration of this topic, see \cite{ap_jin, ap_jin2, Neg13, Deg13}. These AP schemes are often based on Implicit-Explicit (IMEX) time integrators, which have a long history of development and have benefited from numerous contributions. Important references on IMEX methods include \cite{ARW95, HW96, AscherRuuthSpiteri, KENNEDY2003139, PR01, BFR16, dimarco_imexrk}.

The primary goal of this paper is to design a class of high-order schemes for the Euler-Poisson system that are uniformly stable with respect to the Debye length and degenerate into consistent and accurate discretizations of the quasineutral system in the limit as the Debye length approaches zero. Although significant progress has been made in developing Asymptotic-Preserving (AP) schemes in the quasineutral limit \cite{CDV_ASP_05, CDV07, Deg13, ACS24, ALVAREZLAGUNA2020109634}, few contributions have focused on the development of high-order AP schemes.
Another key contribution of this work is the ability of the proposed schemes to handle very general initial conditions. Most of the existing studies on the quasineutral limit assume that the initial conditions are well-prepared, meaning they are compatible with the asymptotic model. However, this assumption can be quite restrictive in practice, as one would expect the numerical scheme to be capable of handling a broad range of initial conditions. In this paper, we explore the challenges posed by the choice of general initial conditions, the implications they have for the numerical solution, and we propose a strategy that ensures stability and consistency with the limit system, even when arbitrarily chosen initial conditions are used.

The strategy employed to ensure the above discussed properties is primarily based on an Implicit-Explicit (IMEX) time integration approach combined with a suitable penalization technique. Following the ideas from \cite{ALVAREZLAGUNA2020109634} and several works focused on the low Mach limit in fluid dynamics \cite{boscarino, BQR19, AS20}, an implicit scheme for the density $\rho$ can be derived by utilizing the divergence of the momentum equation, with the introduction of a penalization term, $\pm \nabla \phi$. This approach allows for the projection of the density onto the quasineutral state, enabling the Poisson equation to be solved in a stable manner.
On this path, we first introduce such a penalization technique and then extend the first-order approach through the use of IMEX Runge-Kutta methods. We successively analyze how a particular class of IMEX schemes \cite{dimarco_imexrk, diffusion1, diffusion2}, known as type A IMEX Runge-Kutta methods, ensures the Asymptotic-Preserving (AP) property for a broad class of initial conditions. This stands in contrast to other classes of methods, such as the CK IMEX Runge-Kutta schemes, which fail to handle initial data that are not consistent with the limit system. Moreover, the latter class of methods will be shown to be unstable in certain situations. 

The rest of the paper is organized as follows. In the next section, the Euler-Poisson model 
and its quasi-neutral limit are recalled. Then, the new time discretization is presented. In Section 4, the asymptotic preserving property is proved for arbitrarily initial data. In the same section. we illustrates the limits of particular class of IMEX schemes and we show that they not be able to provide the required asymptotic preserving property and stability. Section 5 details the space approximation and finally, in Section 6, some numerical results are proposed to illustrate the capability of the new scheme compared to other schemes from the literature. A final section permits to draw some conclusions and suggest future investigations.

\section{One-fluid Euler-Poisson System and its quasi-neutral limit}
\label{sec:ep_qn_model}
We consider in the following the Euler-Poisson (EP) equations. This system describes the time evolution of a fluid of electrons in a plasma where uniform
ion background is supposed. The non-dimensional Euler-Poisson 
system \cite{Deg13} reads as:
\begin{eqnarray}
\partial_t \rho + \nabla \cdot q &=& 0, \label{eq:ep_nd_mass}\\ 
\partial_t q + \nabla \cdot (\rho u\otimes u + p(\rho)I) &=& 
\rho\nabla\phi, \label{eq:ep_nd_mom}\\ 
\lambda^2\Delta \phi&=&\rho-1, \;\; x\in \mathbb{T},
\label{eq:ep_nd_poi} 
\end{eqnarray}
where $\rho(t, x)$ is the density, $q(t,x) := (\rho u)(t, x)$ 
is the momentum ($u$ being the mean velocity), $p(\rho) = \rho^{\gamma}$,
$\gamma \geq 1$ the isentropic pressure and $I$ denotes the identity matrix. 
The electric field $E$ obtained as the negative gradient of the electric 
potential $\phi(t,x)$, i.e. $E = -\nabla \phi$. The spatial differential operators
$\nabla$ and $\Delta$ used are the gradient and the Laplacian, respectively. 

The system is supplemented with initial conditions $(\rho, q)(t=0, x)$, while periodic boundary conditions are considered in space. The uniform ion background is represented by the constant
density $1$ in the Poisson equation \eqref{eq:ep_nd_poi}. The scaled Debyelength
$\lambda$ in \eqref{eq:ep_nd_poi} is microscopic in nature and makes the
Euler-Poisson system \eqref{eq:ep_nd_mass}-\eqref{eq:ep_nd_mom} singular in the limit of 
$\lambda \to 0$; see \cite{ACS24, Deg13}. 


In the quasi-neutral limit where there is no charge imbalance, i.e. the limit 
of $\lambda \to 0$ the Poisson equation \eqref{eq:ep_nd_poi} degenerates into 
$\rho = 1$, which is a singular limit. We end this section with a presentation 
of the asymptotic limit system obtained from the limit $\lambda \to 0$ in the 
EP system \eqref{eq:ep_nd_mass}-\eqref{eq:ep_nd_poi}, the quasi-neutral limit system. 
We refer the reader to \cite{CDV_ASP_05, Deg13} for more details, and to \cite{CG00} 
for theoretical aspects associated to the quasi-neutral limit.  
Let us assume that the variables $\rho, q$ and $\phi$ yield the following 
limits as $\lambda \to 0$:
\begin{equation}
  \label{ch:qn_limit_all}
  \lim_{\lambda \to 0} \rho = \rho_{(0)} =1, \quad 
  \lim_{\lambda \to 0 } q = q_{(0)} = \rho_{(0)} u_{(0)}, \quad
  \lim_{\lambda \to 0} \phi  = \phi_{(0)}.
\end{equation}
Hence, from \eqref{eq:ep_nd_mass}-\eqref{eq:ep_nd_poi} taking the limit $\lambda \to 0$ we formally 
obtain the following incompressible Euler system as the quasi-neutral limit of the EP system 
\eqref{eq:ep_nd_mass}-\eqref{eq:ep_nd_poi},
\begin{align}
    \dvg u_{(0)} &= 0, \label{eq:ep_mass_lim} \\
    \Dt u_{(0)} + \nabla\cdot\left(u_{(0)} \otimes u_{(0)} \right)&= 
    \nabla \phi_{(0)}, \label{eq:ep_mom_lim}\\ 
    \rho_{(0)} &=1. \label{eq:ep_poi_lim_}
\end{align} 
In the quasi-neutral limit negative electric potential $ - \phi_{(0)}$ serves as 
the incompressible pressure. Let us remark that the divergence free constraint 
$\nabla\cdot u_{(0)}=0$ can be reformulated. Indeed, taking the divergence of \eqref{eq:ep_mom_lim} and using \eqref{eq:ep_mass_lim} 
give 
$$
\nabla^2:\left(u_{(0)} \otimes u_{(0)} \right)= \Delta  \phi_{(0)},  
$$
so that the quasi-neutral limit system  \eqref{eq:ep_mass_lim}-\eqref{eq:ep_mom_lim} can be reformulated as 
\begin{eqnarray}
\label{qn_equations1}
\partial_t u_{(0)} + \nabla \cdot ( u_{(0)}\otimes u_{(0)} ) &=& \nabla\phi_{(0)}, \\ 
\label{qn_equations2}
\Delta \phi_{(0)}&=& \nabla^2 : ( u_{(0)}\otimes u_{(0)} ). 
\end{eqnarray}
Note that the initial data has to be well-prepared in the following sense for the solution of the 
EP system \eqref{eq:ep_nd_mass}-\eqref{eq:ep_nd_poi} to converge to the solution of the quasi-neutral
system \eqref{eq:ep_mass_lim}-\eqref{eq:ep_poi_lim_}. 
\begin{definition}
\label{def:wp_data}
    A triple $(\rho, q, \phi)$ of solution of the EP system \eqref{eq:ep_nd_mass}-\eqref{eq:ep_nd_poi} is 
    called well-prepared if it admits the following form: 
    \begin{equation}
        \begin{aligned}
            \rho = \rho_{(0)} + \lambda^2 \rho_{(2)}, & \quad q = q_{(0)} + \lambda^2 q_{(2)} \\
            \rho_{(0)} = 1, &\quad \dvg q_{(0)} = 0.
        \end{aligned}
    \end{equation}
\end{definition}
The goal of this work is to design time discretization 
of \eqref{eq:ep_nd_mass}-\eqref{eq:ep_nd_poi} which is uniformly stable 
with respect to $\lambda$ and degenerates when $\lambda\to 0$ towards a consistent discretisation 
of the asymptotic quasi-neutral model \eqref{eq:ep_mass_lim}-\eqref{eq:ep_poi_lim_}.      

\section{Time Discretisation}
\label{sec:TD}

To begin with we will first present a first order in time discretisation and then will propose a 
high order extension in the next subsection. Presentation of the first order discretisation is 
aimed towards providing an easy access into the design and analysis elements of the semi-discrete 
schemes developed in this paper. 
As the proposal of time semi-discrete scheme the basically constituted by 
two steps:
\begin{itemize}
    \item in the first step a simple discretisation of a reformulated EP system is proposed;
    \item in the second step an equivalent reformulated time-semi discrete scheme is obtained 
    from the scheme from the first step, which is equivalent to a fully explicit scheme.
\end{itemize}
Before introducing our strategy, let us consider a reformulation of
the EP system \eqref{eq:ep_nd_mass}-\eqref{eq:ep_nd_poi} which will form the basis for the time 
discretisations to follow. Indeed, the schemes are based on a penalisation procedure which consists 
of reformulating the momentum equation \eqref{eq:ep_nd_mom} as, 
\begin{eqnarray*}
\partial_t q + \nabla \cdot (\rho u\otimes u + p(\rho)I) &=& (\rho-1)\nabla\phi + \nabla\phi. 
\end{eqnarray*}
Note that the above equation is equivalent to the original momentum equation as the addition and 
subtraction of $\nabla \phi$ to it cancels out. With this formulation in hand, we introduce the 
following notations 
\begin{equation}\label{eq:notation_ref}
\begin{aligned}
    U:=(\rho, q)^T, \quad F:=(q\otimes q)/\rho + p(\rho) I, & \quad G_{\mbox{ex}} := (0, \nabla\cdot F - (\rho-1)\nabla\phi)^T, \quad G_{\mbox{im}} := (\nabla\cdot q, -\nabla\phi)^T \\
    & H(U, \phi) := -\lambda^2 \Delta \phi- \rho + 1
\end{aligned}
\end{equation}
so that the EP system \eqref{eq:ep_nd_mass}-\eqref{eq:ep_nd_poi} becomes 
\begin{equation}
\label{imex_reformulation} 
\begin{aligned}
\partial_t U + G_{\mbox{ex}} +  G_{\mbox{im}} &= 0, \\
H(U, \phi) &= 0.
\end{aligned}
\end{equation}
Based on this reformulation, an IMEX-ARK schemes can be written by considering a discretisation in which the mass flux 
and the term $\nabla \phi$ are treated implicitly.  

Before going to the presentation of the general case, let us consider the first order case 
which will serve as a good context to present the construction of the scheme. 

\subsection{First order time discretisation}
\label{first_order_subsection}
To discretise the EP system \eqref{eq:ep_nd_mass}-\eqref{eq:ep_nd_poi} in time, we propose the following first 
order time stepping by utilising the decomposition \eqref{imex_reformulation} of the EP system: 
\begin{equation}\label{eq:pen-IMEX-1st}
\begin{aligned}
\frac{U^{n+1} - U^n}{\Delta t} &= - G_{\mbox{ex}}^n - G_{\mbox{im}}^{n+1}, \\
\lambda^2\Delta \phi^{n+1} &= \rho^{n+1} -1.
\end{aligned}
\end{equation}
which, using the notations $U=(\rho, q)$ and $G_{\mbox{ex}}, G_{\mbox{im}}$ from \eqref{eq:notation_ref}, corresponds to 
\begin{eqnarray} 
\label{eq_rho}
\rho^{n+1} &=& \rho^n - \Delta t \nabla \cdot q^{n+1}, \\
q^{n+1} &=& q^n - \Delta t \nabla \cdot(\rho^n u^n\otimes u^n + p^n I) +\Delta t \rho^n\nabla\phi^{n}-\Delta t \nabla\phi^{n}+\Delta t \nabla\phi^{n+1}, \nonumber\\
\label{eq_q}
&=& q^n - \Delta t \nabla \cdot F^n +\Delta t (\rho^n-1)\nabla\phi^{n}+\Delta t \nabla\phi^{n+1}, \\
\label{ep_discrete}
\lambda^2\Delta \phi^{n+1} &=&\rho^{n+1} -1. 
\end{eqnarray} 
Note the above scheme is semi-implicit in nature and both, the mass and momentum updates have implicit terms present 
in it. In order to simplify the scheme with respect to implicit updates we insert the divergence of \eqref{eq_q} into \eqref{eq_rho} which leads to: 
$$
\rho^{n+1} = \rho^n - \Delta t \nabla\cdot q^{n+1} = \rho^n -\Delta t  \nabla\cdot q^n + \Delta t^2 \nabla^2 :   F^n + \Delta t^2 \nabla\cdot ( (\rho^n-1)\nabla \phi^n) -\Delta t^2 \Delta \phi^{n+1}. 
$$
Now, using the Poisson equation $\lambda^2\Delta \phi^{n+1} = \rho^{n+1}-1$ and subtracting $1$ on both sides give: 
$$
\rho^{n+1} -1= \rho^n-1 - \Delta t \nabla\cdot q^{n+1} = \rho^n - 1-\Delta t  \nabla\cdot q^n + \Delta t^2 \nabla^2 :   F^n + \Delta t^2 \nabla\cdot ( (\rho^n-1)\nabla \phi^n) -\frac{\Delta t^2}{\lambda^2}( \rho^{n+1}-1), 
$$
from which we deduce $\rho^{n+1}$ 
\begin{eqnarray}
\rho^{n+1}\hspace{-0.3cm} &=&\hspace{-0.3cm} 1+\frac{\lambda^2}{\lambda^2+ \Delta t^2} \Big[ \rho^n - 1 - \Delta t \nabla\cdot q^n + \Delta t^2 \nabla^2 :   F^n + \Delta t^2 \nabla\cdot ( (\rho^n-1)\nabla \phi^n) \Big]\nonumber\\
 \hspace{-0.3cm}&=&\hspace{-0.3cm} 1+\frac{\lambda^2}{\lambda^2+ \Delta t^2} \Big[ \rho^n - 1 - \Delta t \nabla\cdot q^n + \Delta t^2 \nabla^2 :   F^n + \Delta t^2 \nabla\rho^n\cdot \nabla \phi^n+\Delta t^2 (\rho^n -1)\Delta \phi^n  \Big]\nonumber\\
 \label{updaterho}
 \hspace{-0.3cm} &=&\hspace{-0.3cm} 1+\frac{\lambda^2}{\lambda^2+ \Delta t^2} \Big[ \rho^n - 1 - \Delta t \nabla\cdot q^n + \Delta t^2 \nabla^2 :   F^n + \Delta t^2 \nabla\rho^n\cdot \nabla \phi^n\Big]  + \frac{\Delta t^2}{\lambda^2+ \Delta t^2} (\rho^n-1)^2. 
\end{eqnarray} 
Under some assumptions on the initial conditions, 
$\rho^{n+1}$ is updated from \eqref{updaterho} 
so that we can prove $\rho^{n+1}=1+{\cal O}(\lambda^2)$. 
Of course, such assumptions are referred to as well-prepared initial conditions; see Definition~\ref{def:wp_data},
which we precisely, would like to overcome. This is an important aspect and will be discussed later. So, the
Poisson equation can be solved for $\phi^{n+1}$:
\begin{eqnarray}
\Delta  \phi^{n+1}\hspace{-0.3cm}&=&\hspace{-0.3cm}\frac{\rho^{n+1}-1}{\lambda^2}\nonumber\\
\label{updatephi}
\hspace{-0.3cm}&=&\hspace{-0.3cm} \frac{1}{\lambda^2+ \Delta t^2} \Big[ \rho^n - 1 - \Delta t \nabla\cdot q^n + \Delta t^2 \nabla^2 :   F^n + \Delta t^2 \nabla\rho^n\cdot \nabla \phi^n\Big]  + \frac{\Delta t^2}{\lambda^2+ \Delta t^2} \frac{(\rho^n-1)^2}{\lambda^2},
\end{eqnarray} 
and finally the momentum equation can be updated explicitly
\begin{equation} 
\label{updateq}
q^{n+1} =q^n - \Delta t \nabla \cdot F^n +\Delta t (\rho^n-1)\nabla\phi^{n}+\Delta t \nabla\phi^{n+1}.  
\end{equation} 
One has to check that the scheme enjoys the AP property. 
\begin{proposition}
Let us assume that the initial conditions i.e. 
$$
\rho^0=\rho(t=0), \qquad q^0=q(t=0), \quad \mbox{and} \quad \phi^0 = \phi(t=0),
$$
are well-prepared in the sense of Definition~\ref{def:wp_data} and $\lambda^2\phi^0 = \rho^0-1$. 
Then, the time semi-discrete scheme \eqref{updaterho}-\eqref{updateq} initialised with the above initial conditions enjoys the following properties when 
$\lambda\to 0$  
$$
\rho^{n} = 1+{\cal O}(\lambda^2), \quad \nabla\cdot u^{n}={\cal O}(\lambda^2), \quad \phi^{n}={\cal O}(1), \;\;\; n\geq 1. 
$$
Moreover, under the same conditions on the initial conditions, the scheme \eqref{updaterho}-\eqref{updateq} degenerates when $\lambda\to 0$ towards 
$$
u^{n+1} = u^n -\Delta t \nabla\cdot (u^n\otimes u^n) + \Delta t \nabla \phi^{n+1}, \;\;\;\;\; \Delta  \phi^{n+1}= \nabla^2 : (u^n\otimes u^n),  
$$
which is a first order discretisation of the asymptotic limit model \eqref{qn_equations1}-\eqref{qn_equations2}. 
\end{proposition}
\begin{proof}
Under the well-prepared initial conditions assumption, it is obvious from \eqref{updaterho} that $\rho^{n+1}=1+{\cal O}(\lambda^2)$. 
To prove $\nabla\cdot u^{n+1} = {\cal O}(\lambda^2)$, 
let us take the divergence of \eqref{updateq} to get 
\begin{eqnarray*}
\nabla \cdot q^{n+1} &=& \nabla \cdot q^n - \Delta t \nabla^2 :  F^n +\Delta t  \nabla \cdot \Big[(\rho^n-1)\nabla\phi^{n}\Big]+\Delta t \Delta\phi^{n+1}. 
\end{eqnarray*}
Using well-prepared initial conditions and the expression for $F$ from \eqref{eq:notation_ref} leads to 
\begin{eqnarray}
\label{check_divu}
\nabla \cdot q^{n+1} &=&   {\cal O}(\lambda^2) - \Delta t \nabla^2 :  (u^n\otimes u^n) +\Delta t \Delta\phi^{n+1}. 
\end{eqnarray}
Now, let us consider \eqref{updatephi}, which degenerates when $\lambda\to 0$ (still under well-prepared initial conditions) into  
\begin{equation}
\label{limit1_phi}
\Delta  \phi^{n+1}= \nabla^2 : (u^n\otimes u^n) +  {\cal O}(\lambda^2),  
\end{equation}
so that, substituting \eqref{limit1_phi} into \eqref{check_divu} 
leads to $\nabla\cdot q^{n+1} = \nabla\cdot u^{n+1} =  {\cal O}(\lambda^2)$. Finally, going back to \eqref{updateq} and using the previous limits gives
\begin{equation}
\label{limit1_u}
u^{n+1} = u^n - \Delta t \nabla \cdot (u^n\otimes u^n) + \Delta t\nabla\phi^{n+1}. 
\end{equation}
\end{proof}
To conclude, under well-prepared initial conditions assumption, this first order time semi-discrete scheme \eqref{updaterho}-\eqref{updateq} degenerates into a scheme consistent of the expected asymptotic model \eqref{qn_equations1}-\eqref{qn_equations2}. If the well-prepared initial conditions assumption is not satisfied, it is easy to see from \eqref{updaterho}, the last term $\Delta \phi^n$ is stiff (typically, $\Delta \phi^n = {\cal O}(1/\lambda^2)$ if $\rho^n$ is not well-prepared), so that $\rho^{n+1}$ is not close to $1$. In the next subsection, we will extend this first order discretisation to high order and will see that type-A IMEX-RK  schemes \cite{AscherRuuthSpiteri, PR01}, enable to overcome the well-preparedness assumption.
\begin{remark}
The order in which the unknowns are updated is different from the strategy proposed by \cite{ACS24, CDV_ASP_05, CDV07}. Indeed, in these works, thanks to a reformulated Poisson equation, the electric potential is first computed so that the momentum can be updated explicitly and, finally an explicit update as well for the density. Whereas, in the method proposed in the present work, the density is first computed explicitly, then the potential by solving the usual Poisson equation and finally an explicit update for the momentum. It shares some similarities with the work in \cite{ALVAREZLAGUNA2020109634}. 
\end{remark}

\subsection{IMEX-RK and indirect approach for DAEs of index 1}
The time discretisation approach \eqref{eq:pen-IMEX-1st} presented in the previous section is just first order accurate and can be generalised to high order in time by using a combination of IMEX-RK \cite{AscherRuuthSpiteri, PR01} schemes and indirect approach for index-1 DAEs \cite{hairer-wanner-2}. As a consequence we shall see, the subcategory of ARS called type-A schemes enables to overcome the restrictive well-prepared initial conditions assumption. We need to note that IMEX-RK schemes \cite{AscherRuuthSpiteri, PR01} are typically well suited for time discretisation of evolution equations of the form:
\begin{eqnarray*}
    \Dt V + \nabla \cdot \mathcal{F}(V) = 0.
\end{eqnarray*}
Whereas, the EP system \eqref{eq:ep_nd_mass}-\eqref{eq:ep_nd_poi} due to the presence of the Poisson equation \eqref{eq:ep_nd_poi} has an elliptic equation which doesn't evolve in time explicitly. Therefore, if we consider the first order discretisation \eqref{eq:pen-IMEX-1st} presented in the previous subsection, an elliptic problem was solved for, at each time step. The indirect approach for DAEs is particularly required to generalise the imposition of the elliptic problem appropriately, see \cite{ACS24} for a similar treatment. 

In this section what follows from now, we first recall the IMEX-RK schemes and its two subcategories called type-CK and type-A \cite{boscarino, dimarco_imexrk}. We also succintly present the indirect approach for index-1 DAEs. Finally, high order AP schemes are presented for the Euler-Poisson system in the quasi-neutral limit, utilising the design philosophy of \eqref{eq:pen-IMEX-1st}.   

\paragraph{IMEX-RK Time Discretisation\\}
IMEX-RK schemes provide a robust and efficient framework to design AP schemes for singular perturbation problems. In this work, we only consider a subclass of the IMEX-RK schemes, namely diagonally implicit (DIRK) schemes. An $s$-stage IMEX-RK scheme is characterised by the two $s\times s$ lower triangular matrices $\tilde{A} = (\tilde{a}_{i,j})$, and $A=(a_{i,j})$, the coefficients $\tilde{c}=(\tilde{c}_1,\tilde{c}_2,\ldots,\tilde{c}_s)$ and $c=(c_1, c_2, \ldots, c_s)$, and the weights $\tilde{\omega}=(\tilde{\omega}_1, \tilde{\omega}_2,
\ldots,\tilde{\omega}_s)$ and $\omega = (\omega_1, \omega_2, \ldots, \omega_s)$. Here, the entries of $\tilde{A}$ and $A$ satisfy the conditions $\tilde{a}_{i,j}=0$ for $ j \geq i$, and $a_{i,j}=0$ for $j>i$. Let us consider the following stiff system of ODEs in an additive form:   
\begin{equation}
  \label{eq:stiff_ODE}
  y^\prime = f(t,y) + \frac{1}{\veps} g(t,y),
\end{equation}
where $0<\veps\ll 1$ is called the stiffness parameter. The functions $f$ and $g$ are known as, respectively, the non-stiff part and the stiff part of the system \eqref{eq:stiff_ODE}; see, e.g.\ \cite{hairer-wanner-2}, for a comprehensive treatment of such systems. 

Let $y^n$ be a numerical solution of \eqref{eq:stiff_ODE} at
time $t^n$ and let $\Dlt$ denote a fixed time-step. An $s$-stage
IMEX-RK scheme, cf.,\ e.g.\ \cite{AscherRuuthSpiteri, PR01}, updates $y^n$ to $y^{n+1}$ through $s$ intermediate stages:  
\begin{align}
  Y^{(i)} &= y^n + \Dlt \sum\limits_{j=1}^{i-1}\tilde{a}_{i,j} f(t^n + \tilde{c}_j\Dlt, Y^{(j)}) + 
        \Dlt\sum \limits_{j=1}^s a_{i,j} \frac{1}{\veps}g(t^n + c_j
        \Dlt,Y^{(j)}), \ 1 \leq i \leq s, \label{eq:imex_Yi} \\
  y^{n+1} &= y^n  + \Dlt \sum\limits_{i=1}^{s}\tilde{\omega}_{i} f(t^n
            + \tilde{c}_i\Dlt, Y^{(i)}) + \Dlt\sum \limits_{i=1}^s
            \omega_{i}\frac{1}{\veps} g(t^n +
            c_i\Dlt,Y^{(i)}), \label{eq:imex_yn+1} 
\end{align}
where the coefficients $\tilde{a}_{i,j}, a_{i,j}$, 
$\tilde{c}_{i}, c_{i}$ and $\tilde{\omega}_{i}, \omega_{i}$ 
are components of the so-called Butcher tableau 
\begin{figure}[htbp]
  \centering
  \begin{tabular}{c|c}
    $\tilde{c}^T$	&$\tilde{A}$\\
    \hline 
			&$\tilde{\omega}^T$
  \end{tabular}
  \quad
  \begin{tabular}{c|c}
    $c^T$	&$A$\\
    \hline 
		&$\omega^T$
  \end{tabular}
  \caption{Double Butcher tableau of an IMEX-RK scheme.}
  \label{fig:butcher_tableau}
\end{figure}
with $\tilde{c}, \tilde{b},c,b\in\mathbb{R}^s$ and $\tilde{A},A \in{\cal M}_{s, s}(\mathbb{R})$. 
As mentioned above we consider two subcategories of IMEX-RK namely, the type-A and type-CK schemes which are defined below; see \cite{boscarino, dimarco_imexrk, KENNEDY2003139} for details. 
\begin{definition}
\label{eq:typeA_CK}
An IMEX-RK scheme is said to be of 
\begin{itemize}
\item type-A, if the matrix $A$ is invertible; 
\item type-CK, if the matrix $A \in \mbb{R}^{s \times s}, \ s \geq 2$,
  can be written as  
\begin{equation*}  
  A =  
  \begin{pmatrix}
    0 & 0 \\
    \alpha & A_{s-1 }
  \end{pmatrix},
\end{equation*}
where $\alpha \in \mbb{R}^{s-1} $ and $A_{s-1} \in \mbb{R}^{s-1 \times
s-1}$ is invertible.
\end{itemize} 
\end{definition}
Another property which is crucial to obtain the AP property is that the IMEX-RK scheme is globally stiffly accurate.
\begin{definition}
\label{eq:GSA}
An IMEX-RK scheme with the Butcher tableau given in
Figure~\ref{fig:butcher_tableau} is said to be globally stiffly
accurate (GSA), if
\begin{equation} 
  \tilde{a}_{s,j} = \tilde{\omega}_{j},  \quad a_{s,j} = \omega_{j}  \ \mbox{for all} \  j = 1,
\ldots, s.  
\end{equation}
\end{definition}

\paragraph{Indirect Approach for RK Methods for DAEs of Index 1\\}
It has to be noted that even though implicit RK schemes were primarily designed for stiff systems of ODEs, these methods can be systematically extended to even broader classes of equations, such as DAEs. What follows is a brief exposure
to the so-called indirect approach \cite{HW96} to design implicit RK schemes for index-1 DAEs. Dedicated to this aim, we consider an arbitrary set DAEs of index 1 in the form    
\begin{align}
  y^{\prime} &= f(y, z), \label{eq:ep_DAE_I1_1}\\
  0 &= g(y,z), \label{eq:ep_DAE_I1_2}
\end{align} 
where $f$ and $g$ are sufficiently smooth functions, and $y$ and $z$ are vectors of appropriate dimensions with respect to the functions $f$ and $g$.

We now consider an $s$-stage implicit RK method defined by the triple $(A,c,\omega)$, where $A=(a_{i,j})$ are the coefficients, $c=(c_j)$ are the intermediate times and $\omega=(\omega_{j})$ are the weights. The indirect approach \cite{HW96} defines an RK scheme for \eqref{eq:ep_DAE_I1_1}-\eqref{eq:ep_DAE_I1_2} with intermediate stages $(Y^{(i)},Z^{(i)})$ in the following way: 
\begin{align}
  Y^{(i)} &= y^n + \Dlt \sum_{j = 1}^s a_{i,j} f(Y^{(j)}, Z^{(j)}),\ i = 1, 2, \ldots, s, \label{eq:ep_Dapp_RK_1} \\
  0 &= g(Y^{(i)}, Z^{(i)}), \ i = 1, 2, \ldots, s\label{eq:ep_Dapp_RK_2}, 
\end{align}
where the last row of the coefficient matrix $A$ is taken to be equal to the vector $\omega$ (ie the implicit RK scheme is stiffly accurate (SA)). Hence, the numerical solution $(y^{n+1}, z^{n+1})$ is defined by $y^{n+1}=Y^s$ and $z^{n+1}=Z^s$
\begin{equation} \label{eq:ch0_eq_cond_dapp_n+1}
  g(y^{n+1}, z^{n+1}) = g(Y^{(s)}, Z^{(s)}) = 0.
\end{equation}
Note that the indirect approach enables us to ensure that the manifold $0 = g(y, z)$  
is preserved by all the intermediate stages and automatically by the
update stage. For a detailed discussion on the derivation, validity
and stability analysis of the direct approach of implicit RK schemes
applied to index-1 DAEs, interested readers may refer to \cite{HW96}.
\begin{remark}
  The manifold preserving property \eqref{eq:ch0_eq_cond_dapp_n+1} of
  the indirect approach plays a crucial role in obtaining a consistent
  discretisation of non-evolutionary equations, such as the Poisson
  equation in the EP system \eqref{eq:ep_nd_mass}-\eqref{eq:ep_nd_poi}. 
\end{remark}
The indirect approach presented in \cite{HW96}, and briefed above, in its core, takes implicit RK schemes for ODEs and derives implicit schemes for DAEs. It becomes computationally very expensive when tackling problems, especially with nonlinearities, such as the EP system  \eqref{imex_reformulation} to use fully implicit solvers.
To gain computational efficiency we optimise the dosage of implicitness combining the indirect approach platform with the IMEX-RK schemes \eqref{eq:imex_Yi}-\eqref{eq:imex_yn+1}.

\subsection{High Order Time Semi-discretisation}
\label{ssec:highopimex}
Inspired by the scheme in \eqref{eq:pen-IMEX-1st}, we employ a combination of the ARS IMEX-RK and the indirect approach to the reformulation \eqref{imex_reformulation} of the EP system to get the following high order extension: 

The intermediate RK stages are, for $i=1, \dots, s,$:
\begin{eqnarray}
\label{ho_scheme1}
U^{(i)} &=& U^n - \Delta t \sum_{j=1}^{i-1} \Big[ \tilde{a}_{i,j}G_{\mbox{ex}}^{(j)} + {a}_{i,j}G_{\mbox{im}}^{(j)})\Big] -\Delta t  {a}_{i,i} G_{\mbox{im}}^{(i)},\\ 
\lambda^2 \Delta \phi^{(i)} &=& \rho^{(i)}-1 \\ 
\label{ho_scheme1_phi}
\end{eqnarray}
and the final update at time $t^{n+1} = t^n + \Dlt$ is:
\begin{eqnarray}
\label{ho_scheme}
U^{n+1} &=& U^n -\Delta t \sum_{j=1}^{s} \Big[ \tilde{\omega}_j G_{\mbox{ex}}^{(j)} + {\omega}_{j} G_{\mbox{im}}^{(j)}\Big], \\
\lambda^2 \Delta \phi^{n+1} &=& \rho^{n+1}-1, 
\label{ho_scheme_phi}
\end{eqnarray}
where the coefficients $\tilde{a}_{i,j}, {a}_{i,j}$ and $\tilde{\omega}_j, {\omega}_j$ are usually prescribed using Butcher tableaux Figures~\ref{fig:butcher_tableau} involving matrices 
$\tilde{A}$ and $A$ as discussed in the previous subsection. As mentioned before we will only use GSA schemes, a consequence of this choice is the last stage $s$  of \eqref{ho_scheme1}-\eqref{ho_scheme1_phi} is the same as \eqref{ho_scheme}-\eqref{ho_scheme_phi}. Therefore, in the sequel we will assume the stage $s$ corresponds to the iteration update i.e. $(U^{n+1}, \phi^{n+1}) = (U^{(s)}, \phi^{(s)})$. Note that the Poisson equation \eqref{ho_scheme1_phi} is a result of the indirect approach.

Now, using the components $U=(\rho, q)$, the scheme \eqref{ho_scheme1}-\eqref{ho_scheme1_phi} can be written as,  for $i=1,\dots, s$   
\begin{eqnarray}
\label{update_rho_ho_pril}
\rho^{(i)} &=& \rho^n - \Delta t \sum_{j=1}^{i-1}  {a}_{i,j} \nabla\cdot q^{(j)}  - \Delta t  {a}_{i,i} \nabla\cdot  q^{(i)} = \widehat{\rho}^{(i)}- \Delta t  {a}_{i,i} \nabla\cdot  q^{(i)},  \\
\label{update_q_ho_nohat}
q^{(i)} &=& q^n - \Delta t \sum_{j=1}^{i-1} \Big[ \tilde{a}_{i,j} \Big(\nabla\cdot F^{(j)} -  (\rho^{(j)}-1)\nabla\phi^{(j)} \Big) - {a}_{i,j} \nabla \phi^{(j)}\Big] + \Delta t  {a}_{i,i}\nabla \phi^{(i)}\\
\label{update_q_ho}
&=& \widehat{q}^{(i)} + \Delta t  {a}_{i,i}\nabla \phi^{(i)},\\
\label{update_phi_ho}
\lambda^2 \Delta \phi^{(i)} &=& \rho^{(i)} - 1, 
\end{eqnarray}
where the notations $\widehat{\rho}^{(i)}$ and $\widehat{q}^{(i)}$ are defined by 
\begin{equation}
\label{def_hat}
\widehat{\rho}^{(i)} =\rho^n - \Delta t \sum_{j=1}^{i-1}  {a}_{i,j} \nabla\cdot q^{(j)}, \;\;\;  \widehat{q}^{(i)} =q^n - \Delta t \sum_{j=1}^{i-1} \Big[ \tilde{a}_{i,j} \Big(\nabla\cdot F^{(j)} -  (\rho^{(j)}-1)\nabla\phi^{(j)} \Big) - {a}_{i,j} \nabla \phi^{(j)}\Big].   
\end{equation}
Written under this form, the scheme looks like an implicit-cost like scheme. In the sequel, we will manipulate the equations  
in the spirit of what was done for the first order case in Subsection \ref{first_order_subsection} and see the scheme does not require any nonlinear solver and has the cost of fully explicit scheme. 

To this end, taking the divergence of the momentum equation \eqref{update_q_ho} gives  
\begin{eqnarray}
\nabla\cdot q^{(i)} &=& \nabla\cdot \widehat{q}^{(i)} + \Delta t a_{i,i} \Delta \phi^{(i)}, \nonumber\\
\label{div_qi}
&\hspace{-1cm}=&\hspace{-0.6cm}\nabla\cdot q^n - \Delta t \sum_{j=1}^{i-1} \Big[ \tilde{a}_{i,j} \Big(\nabla^2 : F^{(j)} -  \nabla\cdot \Big((\rho^{(j)}-1)\nabla\phi^{(j)}\Big) \Big) - {a}_{i,j} \Delta \phi^{(j)}\Big] + \Delta t  {a}_{i,i}\Delta \phi^{(i)}, 
\end{eqnarray}
and inserting the above expression for $\nabla\cdot q^{(i)}$ in the mass update \eqref{update_rho_ho_pril} gives, 
\begin{eqnarray*}
\rho^{(i)} &=& \widehat{\rho}^{(i)}
 - \Delta t  {a}_{i,i} \left\{\nabla\cdot q^n - \Delta t \sum_{j=1}^{i-1} \Big[ \tilde{a}_{i,j} \Big(\nabla^2 : F^{(j)} -  \nabla\cdot \Big((\rho^{(j)}-1)\nabla\phi^{(j)}\Big) \Big) - {a}_{i,j} \Delta \phi^{(j)}\Big] + \Delta t  {a}_{i,i}\Delta \phi^{(i)}\right\}, \nonumber\\
&=&   \widehat{\rho}^{(i)}- \Delta t  {a}_{i,i} \nabla\cdot q^n \nonumber\\
&&+\Delta t^2  a_{i,i}  \sum_{j=1}^{i-1} \left\{ \tilde{a}_{i,j} \Big(\nabla^2 : F^{(j)} -  \nabla\cdot \Big((\rho^{(j)}-1)\nabla\phi^{(j)}\Big) \Big) - {a}_{i,j} \Delta \phi^{(j)}\right\} - \Delta t^2  {a}^2_{i,i}\Delta \phi^{(i)}. \nonumber
\end{eqnarray*}
Using the Poisson equation $\lambda^2 \Delta \phi^{(i)} = \rho^{(i)}-1$ to substitute for $\Delta \phi^{(i)}$ and subtracting $1$ from both sides in the above equation gives, 
\begin{eqnarray}
\rho^{(i)} &=& 1 +\frac{\lambda^2}{\lambda^2+\Delta t^2 {a}^2_{i,i} } \left\{ \rho^n - 1 -\Delta t \sum_{j=1}^{i-1} \tilde{a}_{i,j} \nabla\cdot q^{(j)}- \Delta t  {a}_{i,i} \nabla\cdot q^n \right.\nonumber\\
\label{update_rho_ho}
&&\left.+\Delta t^2  {a}_{i,i}  \sum_{j=1}^{i-1} \left\{ \tilde{a}_{i,j} \Big(\nabla^2 : F^{(j)} -  \nabla\cdot \Big((\rho^{(j)}-1)\nabla\phi^{(j)}\Big) \Big) - {a}_{i,j} \Delta \phi^{(j)}\right\} \right\}. 
\end{eqnarray}
The above version of $\rho^{(i)}$ would serve towards the analysis of the scheme to be presented in the ensuing section. Computationally, the following form using the notations  $\widehat{\rho}^{(i)}$ and $\widehat{q}^{(i)}$ from \eqref{def_hat} is more appropriate:
\begin{equation}
\label{update_rho_ho_comp}
    \rho^{(i)} = 1 + \frac{\lambda^2}{\lambda^2 + \Dlt^2 a_{i. i}^2} \left( \widehat{\rho}^{(i)} - 1 -\dvg \widehat{q}^{(i)} \right)
\end{equation}
\begin{remark}
Note that the first stage only considers the implicit part and as such deserves to be written down separately. Indeed, for $i=1$, \eqref{update_rho_ho} becomes (using $\widehat{\rho}^{(i)}=\rho^n$ and $\widehat{q}^{(i)}=q^n$)
\begin{eqnarray}
\label{update_rho_1st}
\rho^{(1)} &=& 1+ \frac{\lambda^2}{\lambda^2+\Delta t^2 {a}^2_{1,1} }\Big( \rho^n -1 - {a}_{1,1}\Delta t \nabla\cdot q^n\Big)\\
\label{update_q_1st}
q^{(1)} &=& q^n +{a}_{1,1} \Delta t \nabla \phi^{(1)}, \\
\label{update_phi_1st}
\lambda^2\Delta \phi^{(1)} &=& \rho^{(1)}-1. 
\end{eqnarray}
\end{remark}
The scheme can be summarised in Algorithm \ref{algo_epqn} 
\begin{algorithm}\label{algo_epqn}
\caption{High order AP scheme for Euler-Poisson} 
	\begin{algorithmic}[1]
		\For {$i=1,2,\ldots, s$}
\State compute $\rho^{(i)}$ using \eqref{update_rho_ho_comp} 
\State compute $\phi^{(i)}$ using \eqref{update_phi_ho} 
\State compute $q^{(i)}$ using \eqref{update_q_ho_nohat} 
\EndFor
	\end{algorithmic} 
\end{algorithm}
\begin{remark}
    Observe that for the penalised-IMEX scheme the update \eqref{update_rho_ho} of $\rho^{(i)}$  is a fully-explicit evaluation, followed by obtaining the solution of an elliptic problem and then finally the momentum is updated explicitly via \eqref{update_q_ho}. This is equivalent to a standard explicit scheme for the EP system where only one elliptic problem is required to be solved and rest all are explicit updates; see \cite{Deg13, CDV07}. 
\end{remark}

\section{Asymptotic Preserving property}
The goal of this section is to  investigate the behaviour of the penalised-IMEX scheme Algorithm~\ref{algo_epqn} when $\lambda\to 0$, for $\Delta t>0$ fixed. We will consider general type-A schemes as is presented in the previous section. First, we will check if the following constraints are satisfied
\begin{equation}
\rho^n-1={\cal O}(\lambda^2), \;\;\; \nabla\cdot u^n ={\cal O}(\lambda^2), \;\;\; \phi^n={\cal O}(1), \;\;\; n\geq 2,
\end{equation}
which will ensure the compatibility of the numerical solution with the asymptotic model \eqref{eq:ep_mass_lim}-\eqref{eq:ep_poi_lim_}. 
Next, we will prove the penalised-IMEX scheme degenerates as $\lambda\to 0$ towards a high order time discretisation of the asymptotic model \eqref{eq:ep_mass_lim}-\eqref{eq:ep_poi_lim_}. 

\subsection{Preserving the constraints}
In the next proposition, we check the unknown 
of the scheme are correctly projected to the quasi-neutral equilibrium, i.e. $\rho=1$ and $\nabla\cdot u =0$ for general initial conditions.
\begin{proposition} 
\label{prop1}
The  $s$-stage type-A IMEX scheme \eqref{update_rho_ho}-\eqref{update_phi_ho}-\eqref{update_q_ho_nohat} is a time  approximation of \eqref{eq:ep_nd_mass}-\eqref{eq:ep_nd_mom}-\eqref{eq:ep_nd_poi}. Moreover, for any initial conditions $\rho^0=\rho(t=0), q^0=q(t=0)$ with $\lambda^2\Delta \phi^0 = \rho^0-1$, as $\lambda\to 0$, the scheme  satisfies the following properties 
$$
\rho^n-1={\cal O}(\lambda^2), \;\;\; \nabla\cdot u^n ={\cal O}(\lambda^2), \;\;\; \phi^n={\cal O}(1), \;\;\; n\geq 2. 
$$
\end{proposition}

\begin{proof}
Due to the structure of the equations and of the scheme we used, the proof will be separated into two parts corresponding to the first two iterations $n=0$ and $n=1$. These two parts being also divided into two sub-parts corresponding to the first stage $i=1$ and the others $i=2, \dots, s$. 
\paragraph{1. First iteration ($n=0$)} 
We first investigate the behaviour of the scheme after the first iteration. Indeed, for a general initial condition, we shall see that the previous scheme enables the projection of the numerical solution solution to the quasi-neutral space.
\paragraph{First stage $(i=1)$\\}
For the first stage for $\rho, q$ are prescribed as: 
\begin{eqnarray}
\rho^{(1)} &=& \rho^n - \Delta t a_{1,1} \nabla\cdot q^{(1)}\nonumber\\
\label{q1}
q^{(1)} &=& q^n + \Delta t a_{1,1}  \nabla \phi^{(1)}.  
\end{eqnarray}
Now manipulating as done before we get in the first stage for $\rho$ is updated via \eqref{update_rho_1st} as: 
\begin{equation}
\label{rho1}
\rho^{(1)} = 1 + \frac{\lambda^2}{\lambda^2+\Delta t^2a_{1,1}^2}(\rho^n -1 - \Delta ta_{1,1}\nabla\cdot q^n), 
\end{equation}
which ensures $\rho^{(1)}=1+{\cal O}(\lambda^2)$ for any (smooth) initial conditions. Hence, the Poisson equation \eqref{update_phi_ho} becomes 
\begin{equation}
\label{nb_phi}
\Delta \phi^{(1)} = \frac{1}{\lambda^2+\Delta t^2a_{1,1}^2}(\rho^n -1 - \Delta ta_{1,1} \nabla\cdot q^n), 
\end{equation}
which ensures that $\phi^{(1)}={\cal O}(1)$ for any initial conditions.  Taking the divergence of \eqref{q1}  and using 
the expression \eqref{nb_phi} for $\Delta \phi^{(1)}$ give  
$$
\nabla\cdot q^{(1)} = \nabla\cdot q^n + \Delta ta_{1,1} \Delta \phi^{(1)} = \nabla\cdot q^n +  \frac{\Delta ta_{1,1}}{\lambda^2+\Delta t^2a_{1,1}^2}(\rho^n -1 - \Delta ta_{1,1} \nabla\cdot q^n),   
$$
which gives  
\begin{equation}
\label{nb_q}
\nabla\cdot q^{(1)} 
= \Big(\frac{\lambda^2}{\lambda^2+\Delta t^2a_{1,1}^2}  \Big) \nabla\cdot q^n + \frac{\Delta ta_{1,1}}{\lambda^2+\Delta t^2a_{1,1}^2}(\rho^n-1) = {\cal O}(\lambda^2) +  \frac{\Delta t a_{1,1}}{\lambda^2+\Delta t^2a_{1,1}^2}(\rho^n-1).
\end{equation}
We can see the density is projected on the constant $1$ so that the potential is bounded, however, the divergence of the momentum is not projected to $0$ at the end of the first stage if we start with a general (non well-prepared) initial condition. 

In the sequel, we will see that these three constraints $\rho^{(i)}-1={\cal O}(\lambda^2)$, $\phi^{(i)}={\cal O}(1)$ and $\nabla\cdot q^{(i)}={\cal O}(1)$ are propagated through the stages of the first iteration. This will be proved by induction, the first step ($i=1$) being already shown to be true. 

\paragraph{Induction on the stages of the first iteration\\}
Let us assume for the stages $j=1, \dots, i-1$, we have 
\begin{equation}
\label{ia}
\rho^{(j)}-1={\cal O}(\lambda^2), \quad \phi^{(j)}={\cal O}(1),  \quad \nabla\cdot q^{(j)}={\cal O}(1), 
\end{equation}
and then we want to prove that this remains true for the stage $i$. 

Regarding the density update \eqref{update_rho_ho} and using \eqref{update_q_ho}, we get: 
$$
\rho^{(i)} = \widehat{\rho}^{(i)} - \Delta t {a}_{i,i} \nabla\cdot q^{(i)} = \widehat{\rho}^{(i)} - \Delta t {a}_{i,i} \nabla\cdot \widehat{q}^{(i)} - \frac{\Delta t^2 {a}^2_{i,i} }{\lambda^2} (\rho^{(i)}-1), 
$$
so that 
\begin{equation}
\label{nb_rho}
\rho^{(i)} -1= \frac{\lambda^2}{\lambda^2 + \Delta t^2 {a}^2_{i,i} } \Big( \widehat{\rho}^{(i)} -1 - \Delta t {a}_{i,i} \nabla\cdot \widehat{q}^{(i)}  \Big). 
\end{equation}
From the definitions \eqref{def_hat} of $\widehat{\rho}^{(i)}$ and $\widehat{q}^{(i)}$ combined with the induction assumptions \eqref{ia}, we conclude $\rho^{(i)}-1= {\cal O}(\lambda^2)$ and then the elliptic problem \eqref{update_phi_ho} for $\phi^{(i)}$ give $\phi^{(i)}= {\cal O}(1)$. Let us now focus on proving the last property $\nabla\cdot q^{(i)} = {\cal O}(1)$. From divergence expression \eqref{div_qi} and the induction assumptions \eqref{ia}, we easily deduce $\nabla\cdot q^{(i)} = {\cal O}(1)$. 
\paragraph{2. Second iteration ($n=1$)}
Now we consider the second iteration, note that from the first iteration, we have the numerical solution $(\rho^1, q^1, \phi^1)$, at time $t^1$ ($t^1 = t^0 + \Dlt$),  satisfying the following: 
\begin{equation}
\label{ia2_first}
\rho^1-1=\rho^{(s)}-1={\cal O}(\lambda^2), \quad  \phi^1=\phi^{(s)}={\cal O}(1),  \quad \nabla\cdot q^{1}=\nabla\cdot q^{(s)}={\cal O}(1).  
\end{equation}
\paragraph{First stage $(i=1)$\\}
Note that the expressions obtained for $\rho^{(1)}$ in \eqref{rho1} and $\phi^{(1)}$ in \eqref{nb_phi} are 
valid for any iteration $n \geq 0$. Therefore, we get after the first stage of the second iteration (corresponding to $n=1$) 
$$
\rho^{(1)} - 1 = {\cal O}(\lambda^2) \quad \mbox{and} \quad \phi^{(1)} = {\cal O}(1), 
$$
and similarly considering \eqref{nb_q}, the expression for divergence, which also remains valid for any $n \geq 0$, 
verifies  
$$
\nabla\cdot q^{(1)} = \mathcal{O} (\lambda^2) + \frac{\Delta t a_{1,1}}{\lambda^2+\Delta t^2a_{1,1}^2}(\rho^{n = 1} - 1) = \mathcal{O} (\lambda^2) + \frac{\Delta t a_{1,1}}{\lambda^2+\Delta t^2a_{1,1}^2}(\rho^{1} - 1). 
$$
Hence, since $\rho^{1}-1={\cal O}(\lambda^2)$, from the first iteration, as a result we have $\nabla\cdot q^{(1)} = {\cal O}(\lambda^2)$. 
To summarise, we finally obtain 
$$
\rho^{(1)} - 1 = {\cal O}(\lambda^2), \quad \phi^{(1)} = {\cal O}(1), \quad \mbox{and} \quad \nabla\cdot q^{(1)} = {\cal O}(\lambda^2).
$$

\paragraph{Induction on the stages of the second iteration\\}
Let assume for the stages $j=1, \dots, i-1$ we have 
\begin{equation}
\label{ia2}
\rho^{(j)} - 1 = {\cal O}(\lambda^2), \quad \phi^{(j)} = {\cal O}(1), \quad \nabla\cdot q^{(j)} = {\cal O}(\lambda^2), 
\end{equation}
and let us prove this remains true for the stage $i$. 

From the density update \eqref{nb_rho} we have that $\rho^{(i)}-1 = {\cal O}(\lambda^2)$ thanks to the definition \eqref{def_hat} of $\widehat{\rho}^{(i)}$ and $\widehat{q}^{(i)}$, combined with the induction assumption \eqref{ia2}. 
Using the result $\rho^{(i)}-1= {\cal O}(\lambda^2)$ in the Poisson equation \eqref{update_phi_ho} we obtain $\phi^{(i)}= {\cal O}(1)$. 

Let us now focus on the proving last property i.e. $\nabla\cdot q^{(i)} = {\cal O}(\lambda^2)$. From \eqref{div_qi} the divergence expression, and the induction assumptions \eqref{ia2}, we get 
\begin{eqnarray*}
\nabla\cdot q^{(i)} &=& \nabla \cdot \widehat{q}^{(i)} + \Delta t {a}_{i,i} \Delta \phi^{(i)} \nonumber\\
&=& \nabla \cdot \widehat{q}^{(i)} + \frac{\Delta t {a}_{i,i} }{\lambda^2}(\rho^{(i)}-1) \nonumber\\
&=& \nabla \cdot \widehat{q}^{(i)} + \frac{\Delta t {a}_{i,i} }{\lambda^2}\Big(\widehat{\rho}^{(i)}-1 - {a}_{i,i} \Delta t \nabla\cdot {q}^{(i)}\Big), \nonumber\\
\Big(1+\frac{{a}^2_{i,i}\Delta t^2}{\lambda^2}\Big) \nabla\cdot q^{(i)}  &=& \nabla \cdot \widehat{q}^{(i)} + \frac{\Delta t {a}_{i,i} }{\lambda^2}\Big(\widehat{\rho}^{(i)}-1 \Big). \nonumber
\end{eqnarray*}
Hence we have after using the definition \eqref{def_hat} of $\widehat{\rho}^{(i)}$
\begin{eqnarray*}
\nabla\cdot q^{(i)}  &=&  \frac{\lambda^2}{\lambda^2+{a}^2_{i,i}\Delta t^2} \Big[\nabla \cdot \widehat{q}^{(i)} + \frac{\Delta t {a}_{i,i} }{\lambda^2}(\widehat{\rho}^{(i)}-1) \Big]\nonumber\\
&=&  \frac{\lambda^2}{\lambda^2+{a}^2_{i,i}\Delta t^2} \nabla\cdot \widehat{q}^{(i)} + \frac{{a}_{i,i}\Delta t}{\lambda^2+{a}^2_{i,i}\Delta t^2} (\rho^n-1) - \frac{{a}_{i,i}\Delta t^2}{\lambda^2+{a}^2_{i,i}\Delta t^2} \Big(\sum_{j=1}^{i-1} a_{i,j} \nabla\cdot q^{(j)} \Big). \nonumber
\end{eqnarray*}
First, $\widehat{q}^{(i)}$ is uniformly bounded so does 
$\nabla \cdot \widehat{q}^{(i)}$. 
Second, from \eqref{ia2}, $\nabla\cdot q^{(j)}={\cal O}(\lambda^2)$ for 
$j=1, \dots, i-1$ and $\rho^n-1={\cal O}(\lambda^2)$ thanks to the 
first iteration (let recall $n=1$ here). Hence, 
we conclude the induction: $\nabla\cdot q^{(i)} ={\cal O}(\lambda^2)$. 
\end{proof}
\subsection{Recovering the asymptotic model}
An AP scheme needs to be able to transition from one models (non-asymptotic) to the other (asymptotic limit) seamlessly. In order to do that it is necessary that in the limit of the singular perturbation paramter, which in this case is $\lambda \to 0$ it boils down to a consistent discretisation for the limit system, which in this case the incompressible system \eqref{eq:ep_mass_lim}-\eqref{eq:ep_poi_lim_}. In this part, we check the penalised-IMEX scheme Algorithm~\eqref{algo_epqn}, enjoys the correct behaviour when $\lambda\to 0$, that is the limit of the scheme corresponds to a scheme of the limit model. 
\begin{proposition} 
\label{prop2}
For any initial conditions $\rho^0=\rho(t=0), q^0=q(t=0)$ with $\lambda^2\Delta \phi^0 = \rho^0-1$, as $\lambda\to 0$, the $s$-stage type-A IMEX scheme \eqref{update_rho_ho}-\eqref{update_phi_ho}-\eqref{update_q_ho_nohat} degenerates towards 
$$
u^{(i)} = u^n -\Delta t \sum_{j=1}^{i-1} \tilde{a}_{i,j}\nabla^2 : (u^{(j)}\otimes u^{(j)}) + \Delta t \sum_{j=1}^{i} {a}_{i,j}\nabla \phi^{(j)}, \;\;\; \mbox{and } \quad \nabla\cdot u^{(i)}=0 \quad \mbox{ for } i = 1, \ldots, s , 
$$
which is an $s$-stage IMEX time discretisation of the asymptotic model \eqref{eq:ep_mass_lim}-\eqref{eq:ep_poi_lim_}. 
\end{proposition}
\begin{proof}
From Proposition \ref{prop1}, $\rho^{(i)}-1={\cal O}(\lambda^2)$ (which is true for $i\geq 1$), we get $q^{(i)} = \rho^{(i)} u^{(i)} = u^{(i)}+{\cal O}(\lambda^2)$ and the flux $F^{(i)}=F(\rho^{(i)}, q^{(i)}) = 
\rho^{(i)} u^{(i)} \otimes u^{(i)} + p^{(i)} I$ degenerates towards  $u^{(i)} \otimes u^{(i)} + I$. Let us now derive the asymptotic scheme for updating $u$.  
Using \eqref{update_q_ho} we get: 
\begin{eqnarray}
u^{(i)} &=& u^n - \Delta t \sum_{j=1}^{i-1} \Big[ \tilde{a}_{i,j}  \nabla\cdot (u^{(j)}\otimes u^{(j)})  - {a}_{i,j}\nabla\phi^{(j)}  \Big]  + \Delta t {a}_{i,i} \nabla\phi^{(i)} + {\cal O}(\lambda^2)\nonumber\\
\label{scheme_ap}
&=& u^n - \Delta t \sum_{j=1}^{i-1} \tilde{a}_{i,j}  \nabla\cdot (u^{(j)}\otimes u^{(j)}) + \Delta t \sum_{j=1}^{i}  {a}_{i,j}\nabla\phi^{(j)}+ {\cal O}(\lambda^2), 
\end{eqnarray}
which is an IMEX scheme for the asymptotic QN model \eqref{eq:ep_mom_lim} on $u$. \\ 

Now, let us determine the equation on the electric potential $\phi$. 
To do so, let us start by considering the divergence of the momentum equation \eqref{update_q_ho_nohat} in combination with Proposition~\ref{prop1} to get 
\begin{eqnarray}
- \Delta t a_{i,i} \Delta \phi^{(i)} &=& \nabla\cdot q^n - \Delta t \sum_{j=1}^{i-1} \Big[ \tilde{a}_{i,j} \Big(\nabla^2 : F^{(j)} - \nabla\cdot ( (\rho^{(j)} -1)\nabla\phi^{(j)}) \Big) - {a}_{i,j}\Delta \phi^{(j)}  \Big]+ {\cal O}(\lambda^2)\nonumber\\
\label{proof_phi}
&=& -\Delta t \sum_{j=1}^{i-1}\tilde{a}_{i,j} \nabla^2 : (u^{(j)}\otimes u^{(j)}) + \Delta t \sum_{j=1}^{i-1}{a}_{i,j}\Delta \phi^{(j)} +  {\cal O}(\lambda^2). 
\end{eqnarray}
Then, we proved the asymptotic scheme \eqref{scheme_ap}-\eqref{proof_phi} is a $s$-stage IMEX scheme for the asymptotic model \eqref{eq:ep_mass_lim}-\eqref{eq:ep_mom_lim}. 

\end{proof}

\begin{remark}
    Note that for $i=1$, the expression \eqref{proof_phi} becomes 
$$
\Delta \phi^{(1)} = {\cal O}(\lambda^2), 
$$ 
but for the next stages, $\phi^{(j)}= {\cal O}(1)$, 
as one can see for $i=2$, 
$$
a_{2,2} \Delta \phi^{(2)} = \tilde{a}_{2,2} \nabla^2 : (u^{(1)}\otimes u^{(1)})+{\cal O}(\lambda^2). 
$$
\end{remark}

\subsection{Analysis of the CK IMEX schemes}
In this part, we quickly discuss how type-CK IMEX schemes 
behave in the QN limit for EP system. The only difference with 
type-A scheme lies in the first stage $i=1$ for which type-CK 
IMEX schemes give 
$$
\rho^{(1)} = \rho^n, \qquad q^{(1)} = q^n, \qquad \lambda^2\Delta \phi^{(1)} = \rho^n-1. 
$$
We immediately see the difference with respect to the type-A schemes,
for which the first stage only considers the implicit part and is thus able to ensure the density at the first stage satisfies $\rho^{(1)}-1={\cal O}(\lambda^2)$, 
which in turn ensures for the potential $\phi^{(1)} ={\cal O}(1)$. However, for type-CK, we have $\phi^{(1)} = \phi^n = {\cal O}(1/\lambda^2)$ (if the initial data is not well-prepared) which tends to infinity as $\lambda\to 0$.
\begin{proposition}
\label{prop3}
\begin{itemize}
\item well-prepared initial data definition: $\rho^0-1={\cal O}(\lambda^2)=\nabla\cdot q^0, \phi^0 = {\cal O}(1)$. The $s$-stage CK-IMEX scheme \eqref{update_rho_ho}-\eqref{update_phi_ho}-\eqref{update_q_ho_nohat} degenerates towards a $s$-stage CK-IMEX scheme for the asymptotic model \eqref{eq:ep_mass_lim}-\eqref{eq:ep_mom_lim}.   
\item general initial data definition. The $s$-stage CK-IMEX scheme \eqref{update_rho_ho}-\eqref{update_q_ho}-\eqref{update_phi_ho} satisfies: 
$\phi^{(i)} = {\cal O}(1/\lambda^2).$
\end{itemize}
\end{proposition}

    \begin{proof}
    For the well-prepared case, we follow the lines of the proof of Proposition \ref{prop2}. 

For the general initial condition case, we have $\rho^{(1)}-1 = {\cal O}(1)$ so  we immediately 
    have $\phi^{(1)} = {\cal O}(1/\lambda^2)$ from the Poisson equation. By induction, we can prove these  
    two estimates are propagated along the stages using 
    \eqref{update_rho_ho}.


  \end{proof}

\section{Space Discretisation}
In this section we describe the spatial discretisation strategy which is coupled with the time semi-discrete penalised-IMEX scheme presented in Section~\ref{ssec:highopimex} to obtain a space-time fully discrete numerical scheme. To do this, we follow a strategy as was done in \cite{ACS24}, where, a second order fully discrete scheme for the EP system was developed using a combination of Rusanov fluxes and simple central differencing. Higher order spatial discretisation can be realised by combining high order shock-capturing schemes with the WENO reconstruction and high order central differences, following the footsteps of \cite{BQR19}, where the authors have developed higher order IMEX schemes for the compressible Euler equations in the low Mach number limit. The endeavour behind the design is to be able to obtain a fully discrete scheme which is compatible enough to showcase the properties of the time-semi discrete penalised IMEX scheme presented in Section~\ref{ssec:highopimex}. 

In order to present the fully-discrete scheme, we introduce the
vector $i = (i_1, i_2, i_3)$ where each $i_m$ for $m = 1,2,3 $
represents the space mesh  in direction $m$: $x_m=i_m\Delta x_m$ with $\Delta x_m=L_m/N_m$, $x_m\in [0, L_m]$ and $N_m$ the number of points. We further introduce the
following finite difference and averaging operators: e.g.\ in the 
$x_m$-direction  
\begin{equation}
  \label{eq:ep_deltax1}
    \delta_{x_m}w_{k}=w_{k+\frac{1}{2}e_m}-w_{k-\frac{1}{2}e_m},
  \quad
  \mu_{x_m}w_{k}=\frac{w_{k+\frac{1}{2}e_m}+w_{k-\frac{1}{2}e_m}}{2},
\end{equation}
for any grid function $w_{k}$. From the semi-discrete scheme \eqref{def_hat}, \eqref{update_rho_ho_comp}
\begin{definition}
\label{def:fully_disc_schm}
    The $i^{th}$ stage of the $s$-stage fully discrete penalised-IMEX scheme for the EP system is defined as follows. The explicit variables are computed first as:
\begin{align}
\hspace{-0.6cm} \widehat{\rho}^{(i)}_{k} &= \rho^n_k - \sum_{j = 1}^{i-1} a_{i, j} \sum_{m = 1}^d \nu_m \delta_{x_m} \mathcal{G}_{m, k}^{(i)}, \\
\hspace{-0.6cm}    \widehat{q}^{(i)}_{k} &= q^n_k - \sum_{j = 1}^{i-1} \tilde{a}_{i, j} \left(\sum_{m = 1}^d \nu_m \delta_{x_m} \mathcal{F}_{m, k}^{(j)} 
    +  (\rho^{(j)}_k - 1)  \nu_m \delta_{x_m} \mu_{x_m} \phi_{k}^{(j)} e_m \right) 
    + \sum_{j = 1}^{i-1} a_{i, j} \nu_m \delta_{x_m} \mu_{x_m} \phi_{k}^{(j)} e_m
\end{align}
Following, the above update $\rho^{(i)}_k$ is updated as:
\begin{equation}
     \rho_{k}^{(i)} = 1 + \frac{\lambda^2}{\lambda^2 + \Dlt^2 a_{i,i}^2} \left( \widehat{\rho}^{(i)}_{k} - 1 - \sum_{m = 1}^d \nu_m \mu_{x_m} \delta_{x_m} \widehat{q}_{m,k}^{(i)} \right) 
\end{equation}
Then the implicit solution is obtained by first solving the discrete elliptic problem which is a linear system for $\{\phi^{(k)}_i \}$ given by,
\begin{equation}
        \sum_{m=1}^d  \frac{\delta_{x_m}}{\Delta x_m} \left(\frac{\delta_{x_m}}{\Delta x_m} \phi^{(i)}_{k}\right) =
    \rho_k^{(i)}- 1,
    \label{eq:phi_fd} 
\end{equation}
Followed by the explicit evaluations,
\begin{equation}
    q^{(i)}_{m, k} = \widehat{q}^{(i)}_{m,k} - \Dlt a_{i,i} \nu_m \delta_{x_m} \phi_k^{(i)}, \quad m = 1, 2, \dots, d.  
\end{equation}
Here, the repeated index $m$ takes values in $\{ 1, 2, \dots, d \}$, $\nu_m := \frac{\Dlt}{\Delta x_m}$ denote the mesh ratios and the vector $(\mathcal{G}_m)$ is an approximation of the mass flux $q$, $\mathcal{F}_m$ is an approximation of the $m^{th}$ element of the flux $F$ in \eqref{eq:notation_ref} and the are defined as: 
\begin{equation}
    \begin{aligned}
        \mathcal{G}^{(j)}_{m, k + \frac{1}{2} e_m} &= \frac{1}{2} (q^{(j)}_{m,k + e_m} + q^{(j)}_{m, k}) \\
        \mathcal{F}^{(j)}_{m, k + \frac{1}{2} e_m} &=  \frac{1}{2} \left( F_{m} (U^{(j), +}_{k + \frac{1}{2} e_m}) + F_{m} (U^{(j), -}_{k + \frac{1}{2} e_m})\right) - \frac{\alpha^{(j)}_{m, k + \frac{1}{2} e_m}}{2} \left(q_{m, k + \frac{1}{2} e_m}^{(j), +} - q_{m, k + \frac{1}{2} e_m}^{(j), -} \right),   
    \end{aligned}
\end{equation}
where $F_m$ denotes the $m^{th}$ element of the flux $F=(F_1, \dots, F_m)$. 
Here, for any conservative variable $w$, we have denoted by
$w^\pm_{i+\frac{1}{2} e_m}$, the interpolated states obtained
using the piecewise linear reconstructions. The wave-speeds are 
computed as follows 
\begin{equation}
  \label{eq:wave_speed1}
  \alpha^{(j)}_{m, k
    +\frac{1}{2}e_m}:=2\max\left(\left|\frac{q_{m,k+\frac{1}{2}e_m}^{(j),-}}{\rho_{k+\frac{1}{2}e_m}^{(j),-}}\right|,
    \left|\frac{q_{m,k+\frac{1}{2}e_m}^{(j),+}}{\rho_{k+\frac{1}{2}e_m}^{(j),+}}\right|\right). 
\end{equation}
\end{definition}
The momentum flux $F$ is approximated by a Rusanov-type
flux, and the the mass flux is approximated using simple
central differences. Hence, the eigenvalues of the Jacobians of the part of the flux which is approximated by Rusanov-type flux can be obtained as $\lambda_{m,1}=0, \lambda_{m,2}= \lambda_{m,3}=\frac{q_m}{\rho}$ and $\lambda_{m,4}= 2 \frac{q_m}{\rho}$ and the CFL condition is given by the timestep restriction on $\Dlt$ at time $t^n$: 
\begin{equation}
\label{eq:ep_CFL}
\Dlt\max_{k}\max_m\left(\frac{\left|\lambda_{m,3,k}\right|}{\Delta x_m}, \frac{\left|\lambda_{m,4,k}\right|}{\Delta x_m}\right)=\nu,
\end{equation} 
where $\nu<1$ is the given CFL number. This CFL condition is a crude estimate, for more rigorous results one may refer to \cite{GP16}.
\begin{remark}
    The classical scheme in \cite{CDV07} and its higher order extension presented in \cite{ACS24} are both based on the ARK scheme utilised here as well. These schemes fail to AP by not only by deviating from the asymptotic limit solution for $\lambda \to 0$ but also suffer through stiff stability condition \cite{Fab92} as:
    \begin{equation*}
        \Dlt \leq \lambda.
    \end{equation*}
    Whereas, the penalised-IMEX schemes developed in this paper don't get restricted with the above terminal stability restriction owing to the addition of the penalisation term to the momentum equation and then identifying the implicit and explicit terms for the ARS scheme to be applied.
\end{remark}

\section{Numerical results} 
In this section we present the results of various numerical test cases carried out in order to not only corroborate the theoretically assertion from the previous sections but also to draw inferences to analyse the penalised-IMEX scheme even further. We consider the following two (Figure~\ref{fig:ep_imexDP} and) butcher tableaux of type-A ARS IMEX-RK scheme from \cite{dimarco_imexrk} and one type-CK scheme (Figure~\ref{fig:ep_imexARS}) from \cite{PR01}.
\begin{figure}[htbp]
  \small
  \centering
  \begin{tabular}{c|c c c c}
    0	  & 0		& 0	   & 0	    & 0\\
    $1/3$ & 0		& 0	   & 0	    & 0\\
    1     & 1		& 0    & 0	    & 0\\
    1     & $1/2$   & 0    & $1/2$  & 0\\ 
    \hline 
          & $1/2$   & 0    & $1/2$  & 0
  \end{tabular}
  \hspace{10pt}
  \begin{tabular}{c|c c c c}
    $1/2$	& $1/2$	 & 0	  & 0	  & 0  \\
    $2/3$	& $1/6$	 & $1/2$  & 0	  & 0  \\
    $1/2$	& $-1/2$ & $1/2$  & $1/2$ & 0  \\
    $1$	    & $3/2$  & $-3/2$ & $1/2$ & $1/2$  \\
    \hline 
            & $3/2$  & $-3/2$ & $1/2$ & $1/2$
  \end{tabular} \\
  \vspace{10pt}
\hspace{2pt}
\begin{tabular}{c|c c c c}
    0	  & 0		& 0	    & 0	     & 0\\
    0     & 0		& 0	    & 0	     & 0\\
    1     & 0		& 1     & 0	     & 0\\
    1     & 0       & $1/2$ & $1/2$  & 0\\ 
    \hline 
    0     & 0       & $1/2$ & $1/2$  & 0\\ 
  \end{tabular}
  \hspace{12pt}
  \begin{tabular}{c|c c c c}
    $\gamma$ & $\gamma$  & 0	         & 0	          & 0  \\
    0	     & $-\gamma$ & $\gamma$      & 0	          & 0  \\
    $1$	     & 0         & $1 - \gamma$  & $\gamma$       & 0  \\
    $1$	     & 0         & $1/2$         & $1/2 - \gamma$ & $\gamma$  \\
    \hline 
             & 0         & $1/2$         & $1/2 - \gamma$ & $\gamma$ 
  \end{tabular}
  \caption{Double Butcher tableaux of type-A Additive IMEX schemes. Top: DP1-A(2, 4, 2)and Bottom: DP2-A(2, 4, 2).}
  \label{fig:ep_imexDP}
\end{figure}
\begin{figure}[htbp]
  \centering
  \begin{tabular}{c|c c c}
    0		& 0			& 0			& 0	\\
    $\gamma$	& $\gamma$		& 0 			& 0	\\
    $1$	& $\delta$		& $1 - \delta$		& 0	\\
    \hline 
                & $\delta$	& $1 - \delta$	& $0$
  \end{tabular}
  \hspace{10pt}
  \begin{tabular}{c|c c c}
    $0$	& $0$		& 0			& 0		\\
    $\gamma$	& $0$		& $\gamma$ 		& 0		\\
    $1$	& $0$		& $1 - \gamma$		& $\gamma$	\\
    \hline 
        & $0$		& $1 - \gamma$		& $\gamma$
  \end{tabular}
  \caption{Double Butcher tableaux of type-CK Additive IMEX schemes.ARS
    (2,2,2). Here, $\gamma=1-\sqrt{2}/2$,
    $\sigma = 1 / 2 \gamma$ and $\delta=1 - \sigma $.}    
  \label{fig:ep_imexARS}
\end{figure}
\subsection{A Small Perturbation of a Uniform QN Plasma: One-dimensional}
\label{sec:ep_numer_case_stud_P1}
A quasineutral (QN) state given of comprising of a constant density and constant velocity is considered:
\begin{equation}
\label{qn_state}
\bar{\rho}(x_1)  = 1, \qquad \bar{u}(x_1) = 1 \qquad \bar{\phi}(x_1) = 0
\end{equation} 
The constant value of density combined with the Poisson equation implies that the electric potential vanishes. This QN state is perturbed with, by adding a small (to be specified) cosine perturbation of magnitude $\delta^2$ to the uniform velocity field; see \cite{CDV07}. All the computations for this test case are carried out in the spatial domain $[0,1]$ with periodic boundary conditions for the hydrodynamic variables $\rho$ and $q$ and homogeneous Dirichlet boundaries for the electric potential $\phi$.

The aim of this case study is to test the AP scheme's ability to recover a QN background state independent of the space mesh-size: resolving or not resolving the parameter $\lambda$. To this end we consider two types of perturbations being added to the QN state \eqref{qn_state}, leading to two types of initial data:
\begin{itemize}
    \item Case 1: well-prepared initial conditions; see Definition~\ref{def:wp_data}.  
    \item Case 2: non-well-prepared initial conditions.
\end{itemize}
\paragraph{Case 1\\} To obtain a well-prepared initial data the velocity perturbation is equal to $\delta^2=\lambda^2$ (according to Definition \ref{def:wp_data}) so that the initial condition reads:
\begin{equation}\label{eq:pert_qn_wp}
  \rho(0, x_1) = \bar{\rho}(x_1), \qquad u_1 (0, x_1) = \bar{u}_1(x_1) + \delta^2 \cos (2 K \pi x_1) \qquad \phi (0, x_1) = \bar{\phi}(x_1), 
\end{equation}
where $\bar{\rho}, \bar{u}, \bar{\phi}$ are given by \eqref{qn_state}. 
For all the plots in this subsection the $y$-axis is in log scale.  Figure~\ref{fig:Small_pertur_QN_r_m1} shows the space dependency of the density perturbation ($|\rho(t, \cdot) - 1|$), velocity divergence ($|\dvg u(t, \cdot)|$) and the absolute value of the electric potential ($|\phi(t, \cdot)|$) at the final time $t=0.1$ for the DP2A1, and DP2A2 penalised-IMEX schemes; see Definition~\ref{def:fully_disc_schm} and the LSDIRK scheme of SI-IMEX-RK-DAE scheme from \cite{ACS24}. The computation was carried out on a resolved mesh with $N = 10^4$ (the number of points, $\Delta x=1/N$), at a CFL $= 0.45$ and for the scaled Debye length $\lambda = 10^{-4}$. 
\begin{figure}[htbp]
  \centering
  \includegraphics[height=0.16\textheight]{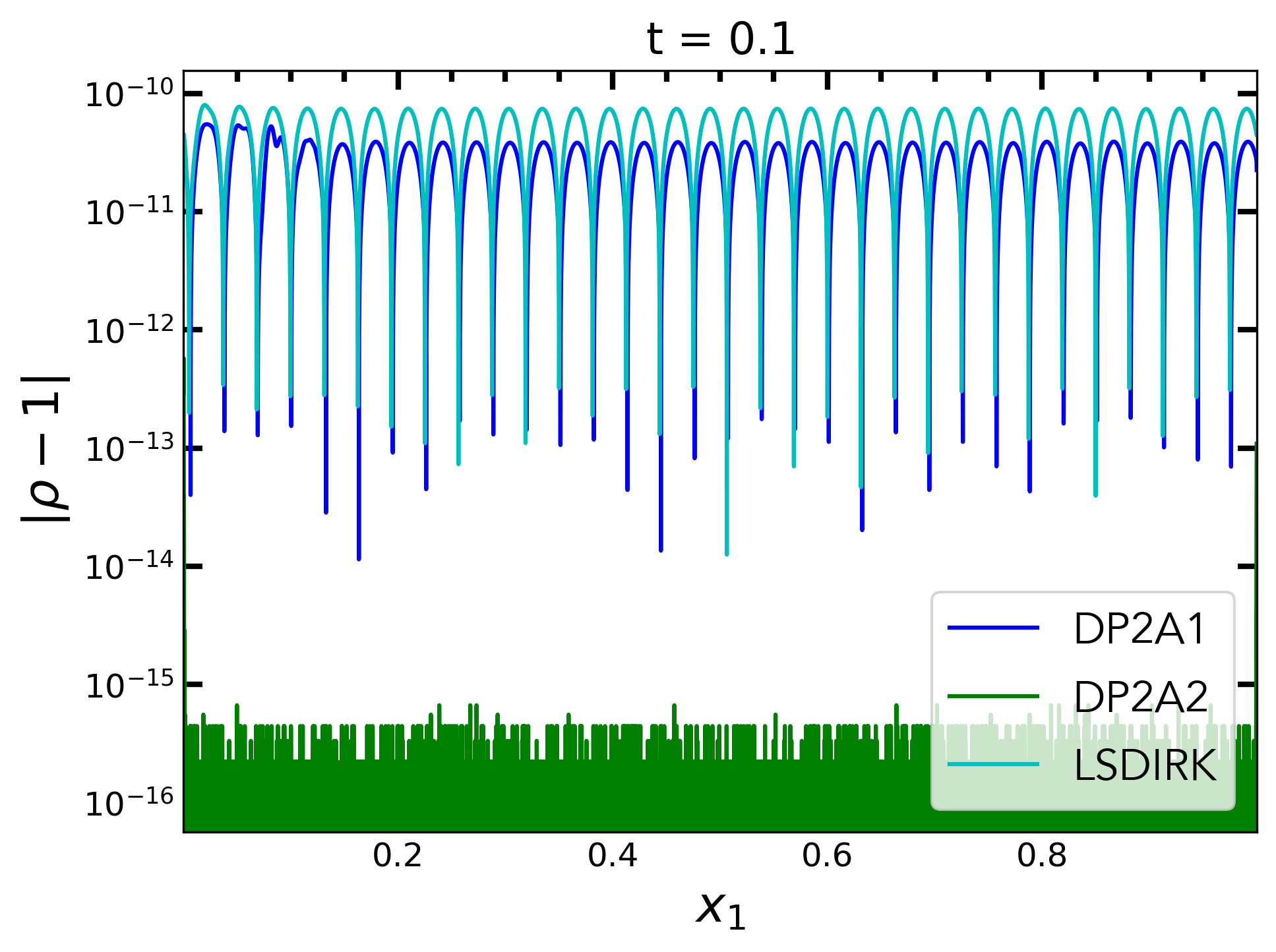}
  \includegraphics[height=0.16\textheight]{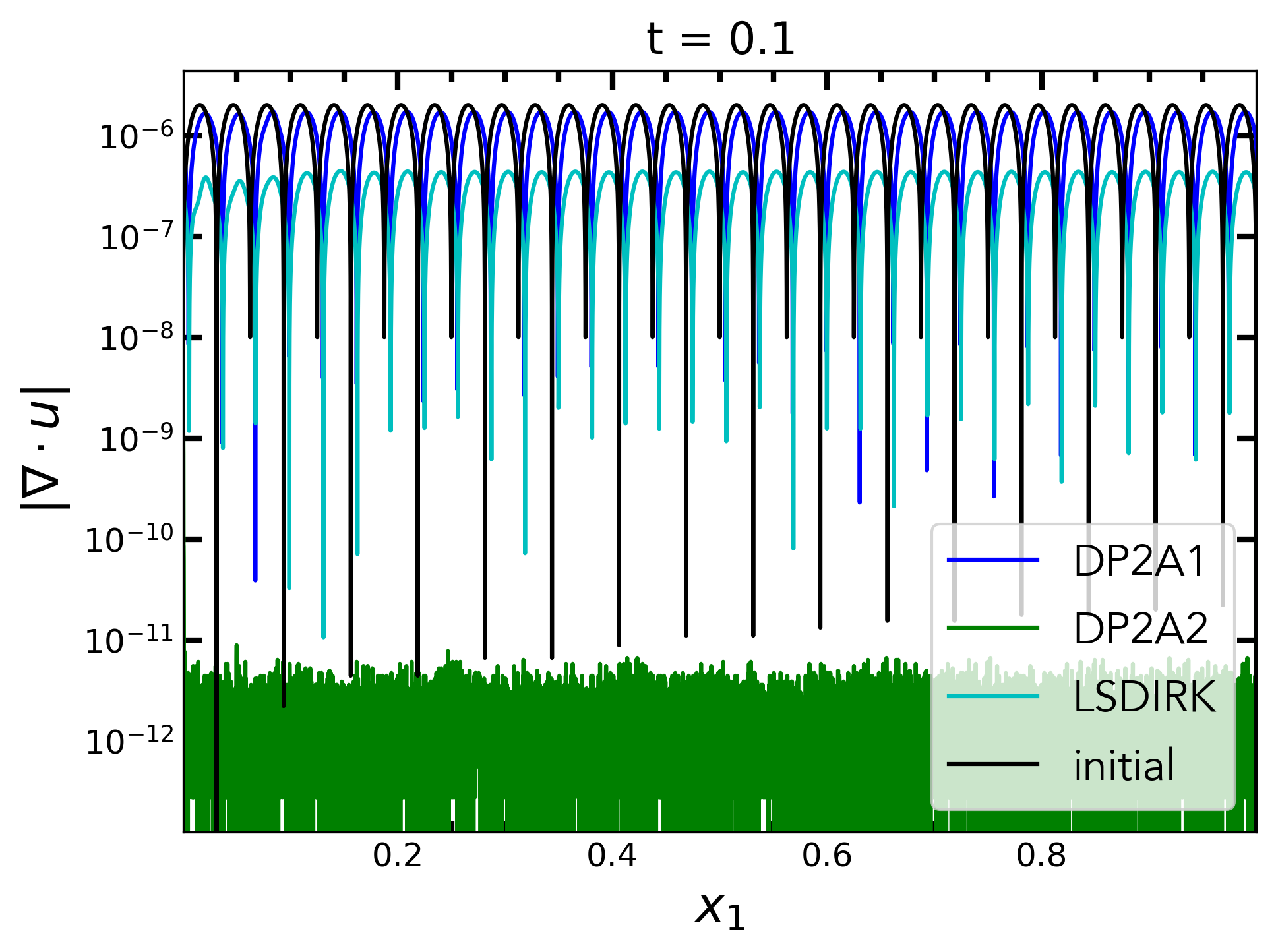}
  \includegraphics[height=0.16\textheight]{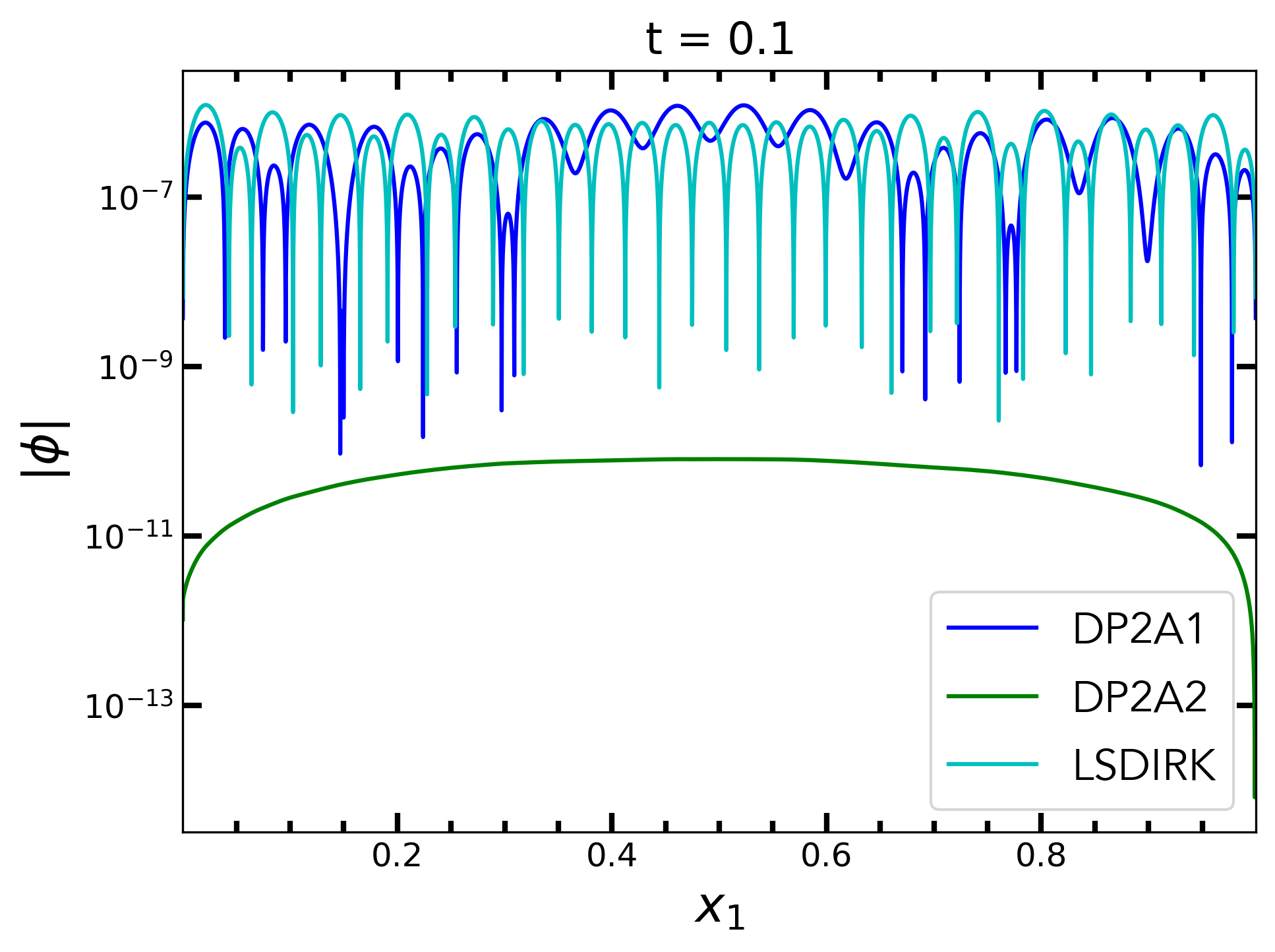}
  \caption{Case 1, resolved mesh ($N=10^4$) 
  with Penalised type-A IMEX and SI-IMEX Schemes. $\lambda = 10^{-4}$. Left: $x_1\to |\rho(t, x_1)-1|$, Center:
    $x_1\to |\dvg u(t, x_1)|$ , Right: $x_1\to |\phi(t, x_1)|$, at time $t
    = 10^{-1}$, $K = 16$.  } 
\label{fig:Small_pertur_QN_r_m1}
\end{figure}
The long time behaviour of all the penalised IMEX schemes on a fine mesh in Figure~\ref{fig:Small_pertur_QN_r_m1} showcases the ability of the penalised IMEX schemes to maintain the well-preparedness of the data for large times. We have plotted the results of the SI-IMEX-RK-DAE scheme to show that the penalised IMEX schemes work at par with them. 
In Figure~\ref{fig:Small_pertur_QN_ur_m1}, the same entities as of Figure~\ref{fig:Small_pertur_QN_r_m1} are plotted with the difference that the computation was carried out on a under resolved mesh with $N = 10^2$ ($\Delta x=1/N$), at a CFL $= 0.45$, for the same scaled Debye length $\lambda = 10^{-4}$. 
\begin{figure}[htbp]
  \centering
  \includegraphics[height=0.162\textheight]{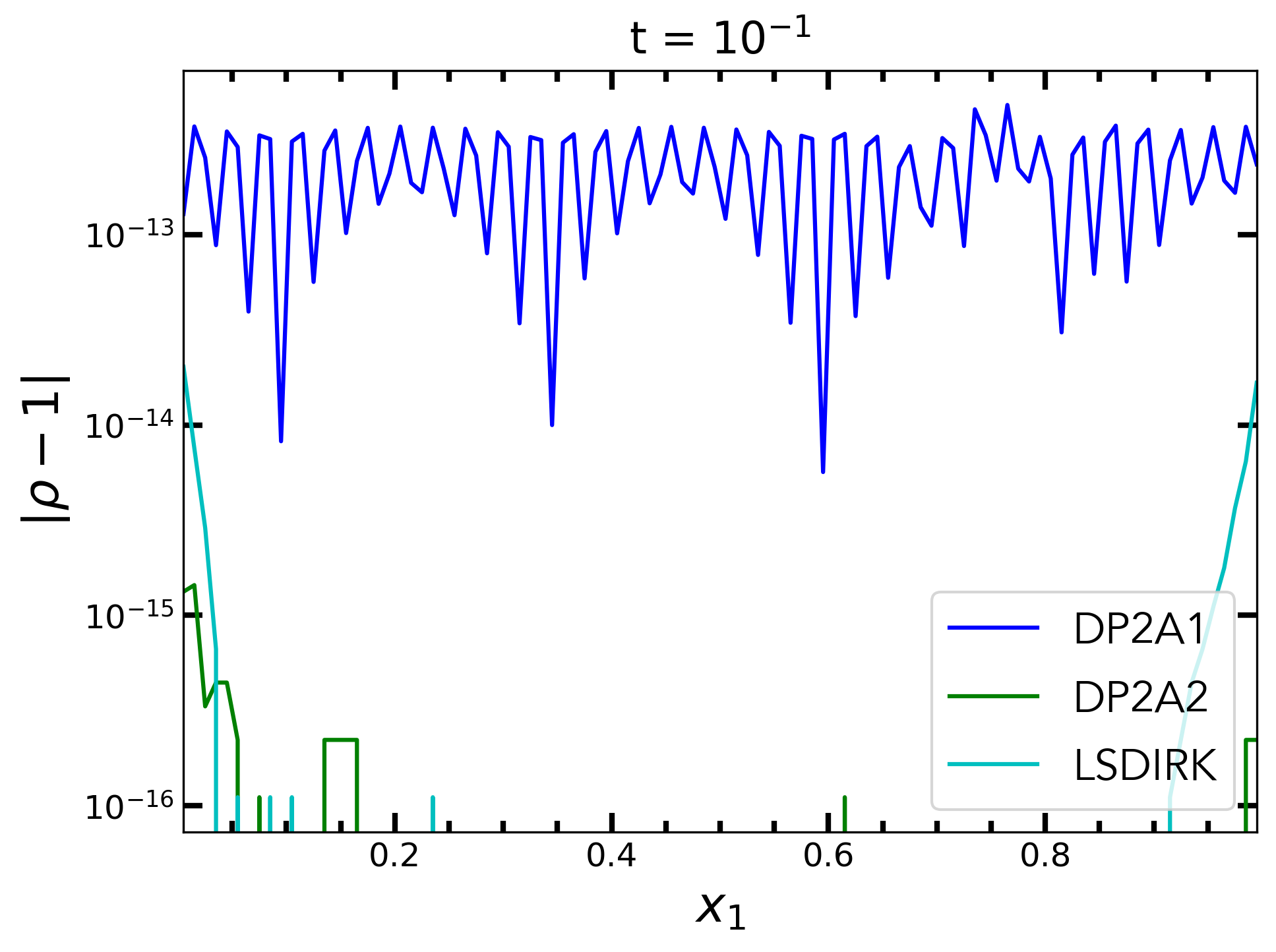}
  \includegraphics[height=0.162\textheight]{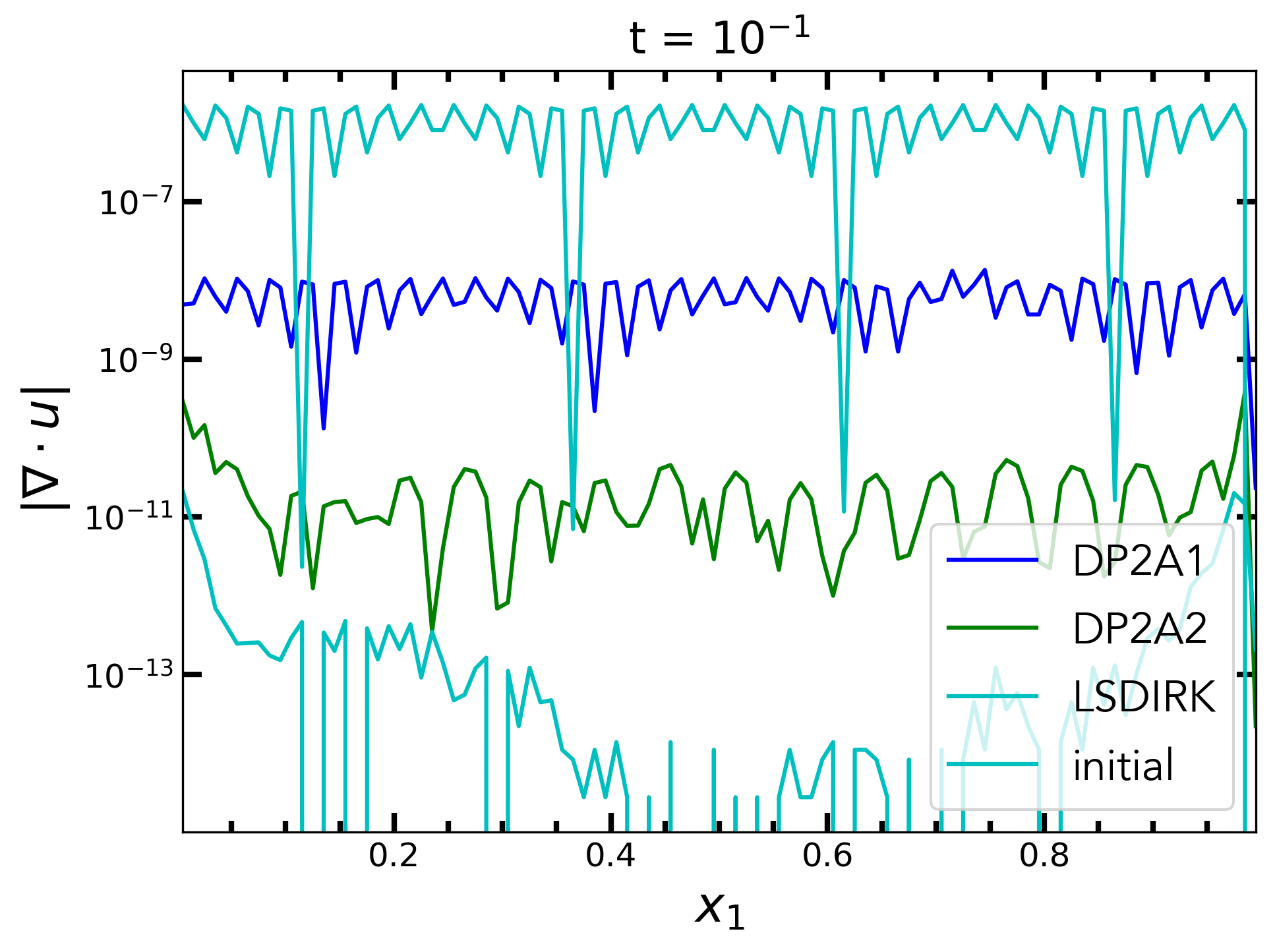}
  \includegraphics[height=0.162\textheight]{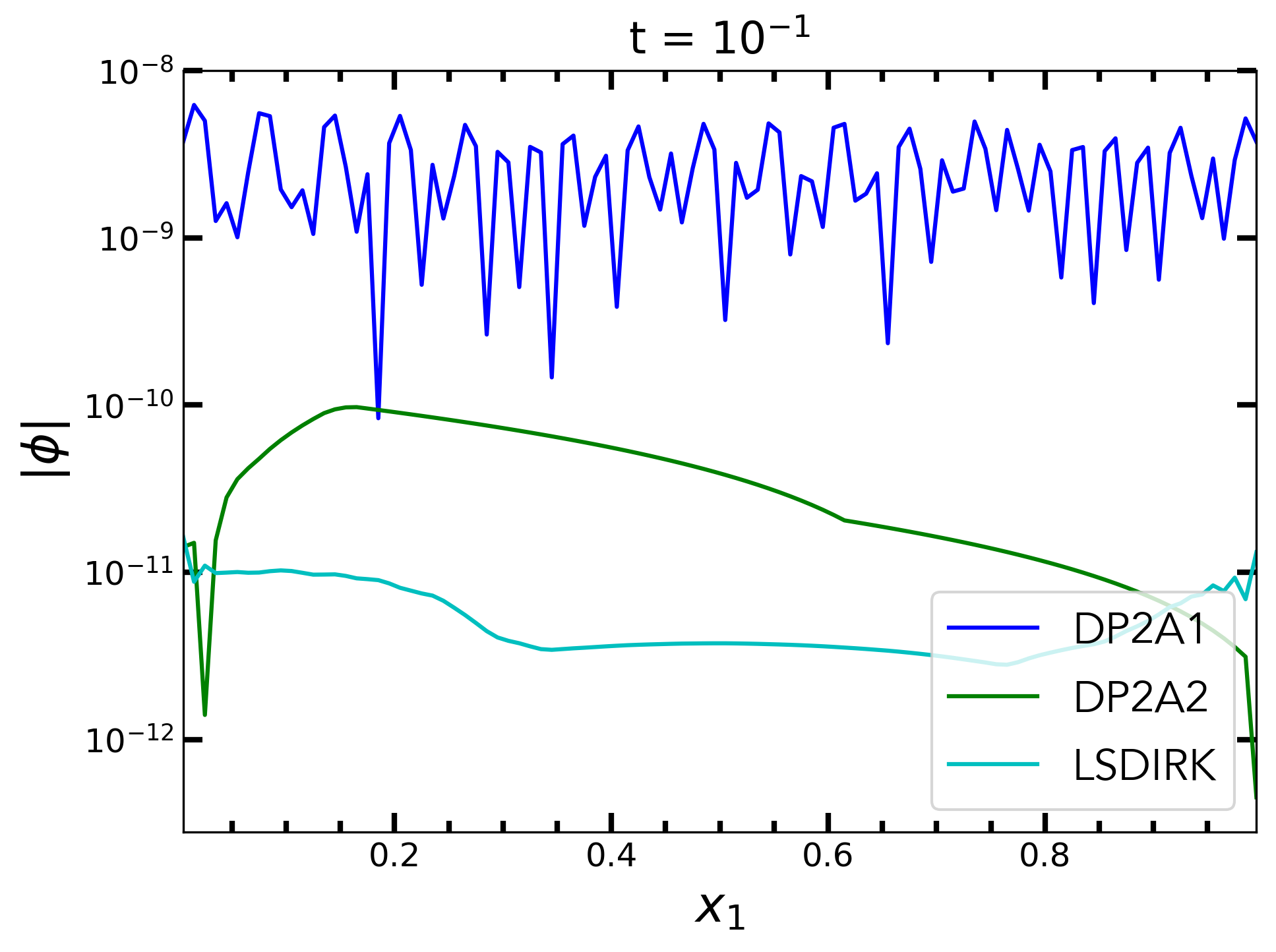}
  \caption{Case 1, under-resolved mesh $N = 10^2$
  with Penalised type-A IMEX and SI-IMEX Schemes. $\lambda = 10^{-4}$. Left: $x_1\to |\rho(t, x_1)-1|$, Center:
    $x_1\to |\dvg u(t, x_1)|$, Right: $x_1\to |\phi(t, x_1)|$, at time $t
    = 10^{-1}$, $K = 16$.}  
  \label{fig:Small_pertur_QN_ur_m1}
\end{figure}
The long time behaviour of all the penalised-IMEX schemes on a very coarse mesh in Figure~\ref{fig:Small_pertur_QN_ur_m1} showcases its ability to maintain the well-preparedness of the data for a long time even on a very coarse mesh. We should note that the asymptotic accuracy of the penalised-IMEX scheme is of paramount importance 
for a coarse mesh as, it effects in reduction in computational cost. The cost reduction stems from two primary branches: first being larger time steps requiring lesser number of time iterations and the second from the elliptic problems, now to be solved on a coarse mesh, in comparison to a fine mesh. To provide a reference the time steps for the penalised-IMEX schemes for the coarse mesh are $\Dlt \sim 0.00225$ whereas for the fine mesh $\Dlt \sim 0.000022$, leading to 46 and 4446 time steps, respectively. 

\paragraph{Case 2\\} Here the velocity is perturbed by the same cosine perturbation as that of Case 1, but with a magnitude of $\delta^2=10^{-2}$ 
in \eqref{eq:pert_qn_wp}. The initial data is not well-prepared in the sense of Definition~\ref{def:wp_data}. 

Figure~\ref{fig:Small_pertur_QN_ur_0p1_nwp} shows the density perturbation, divergence of velocity and the absolute value of potential for the penalised-IMEX schemes at a late time $t = 0.1$ on a coarse mesh $N = 10^{2}$ (not resolving $\lambda=10^{-4}$) with a CFL = 0.25 and Figure~\ref{fig:Small_pertur_QN_ur_1p0_nwp} shows the same at a very late time $t = 1.0$. Looking at both the figures we observe that DP2A2 scheme is able to achieve the equilibrium state quite early, whereas DP2A1 scheme take longer to capture 
the equilibrium. 
\begin{figure}[htbp]
  \centering
  \includegraphics[height=0.162\textheight]{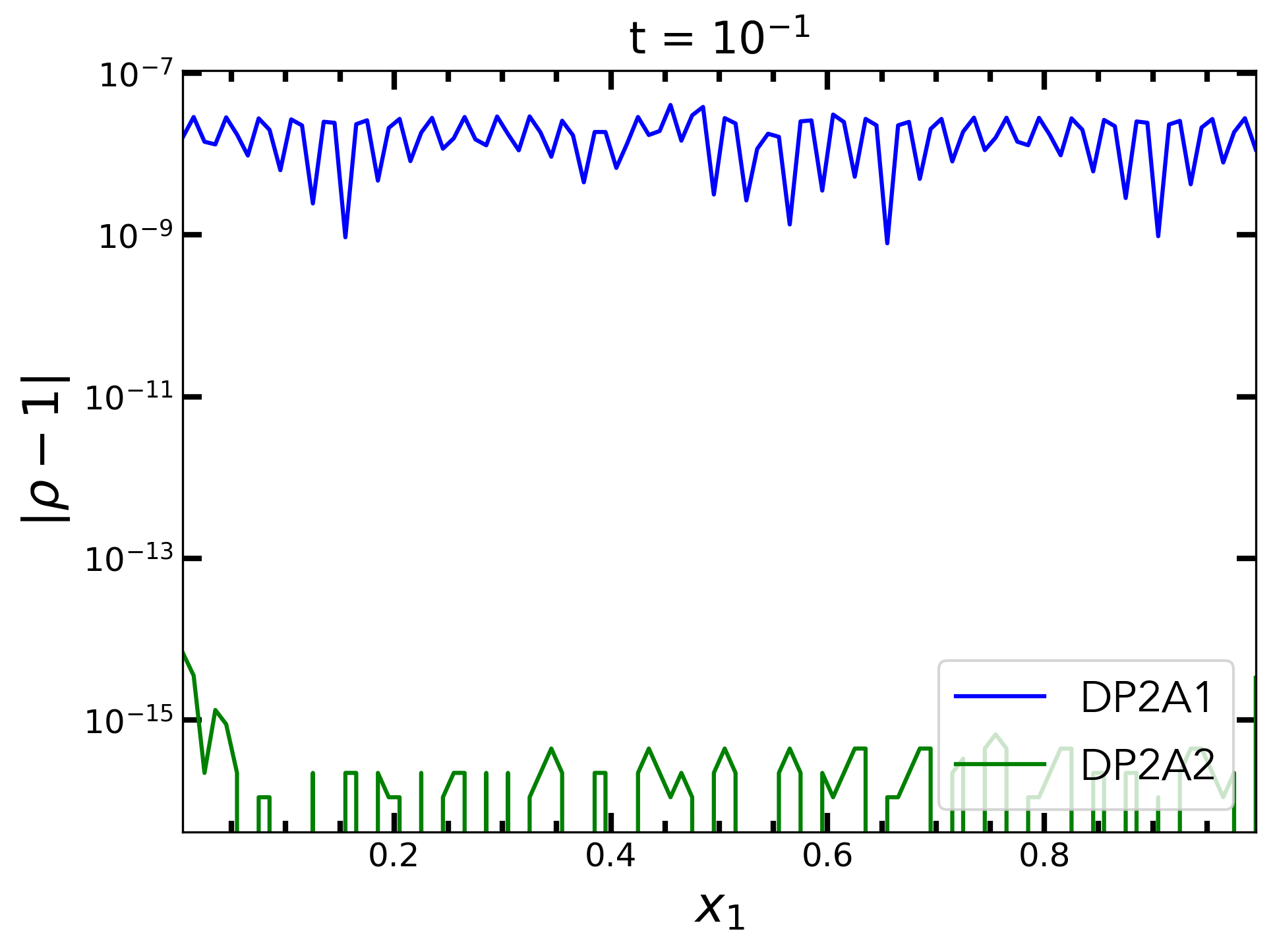}
  \includegraphics[height=0.162\textheight]{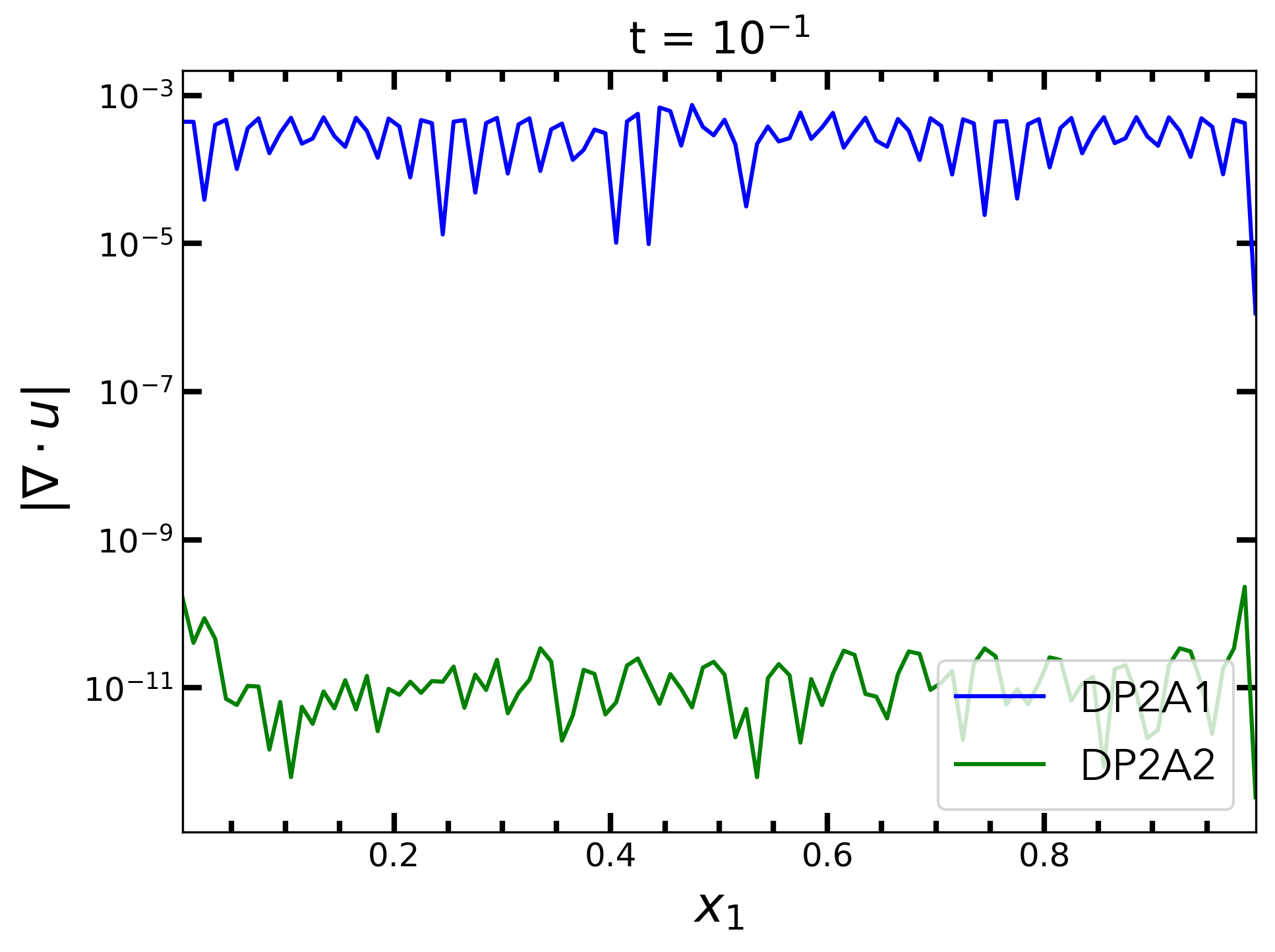}
  \includegraphics[height=0.162\textheight]{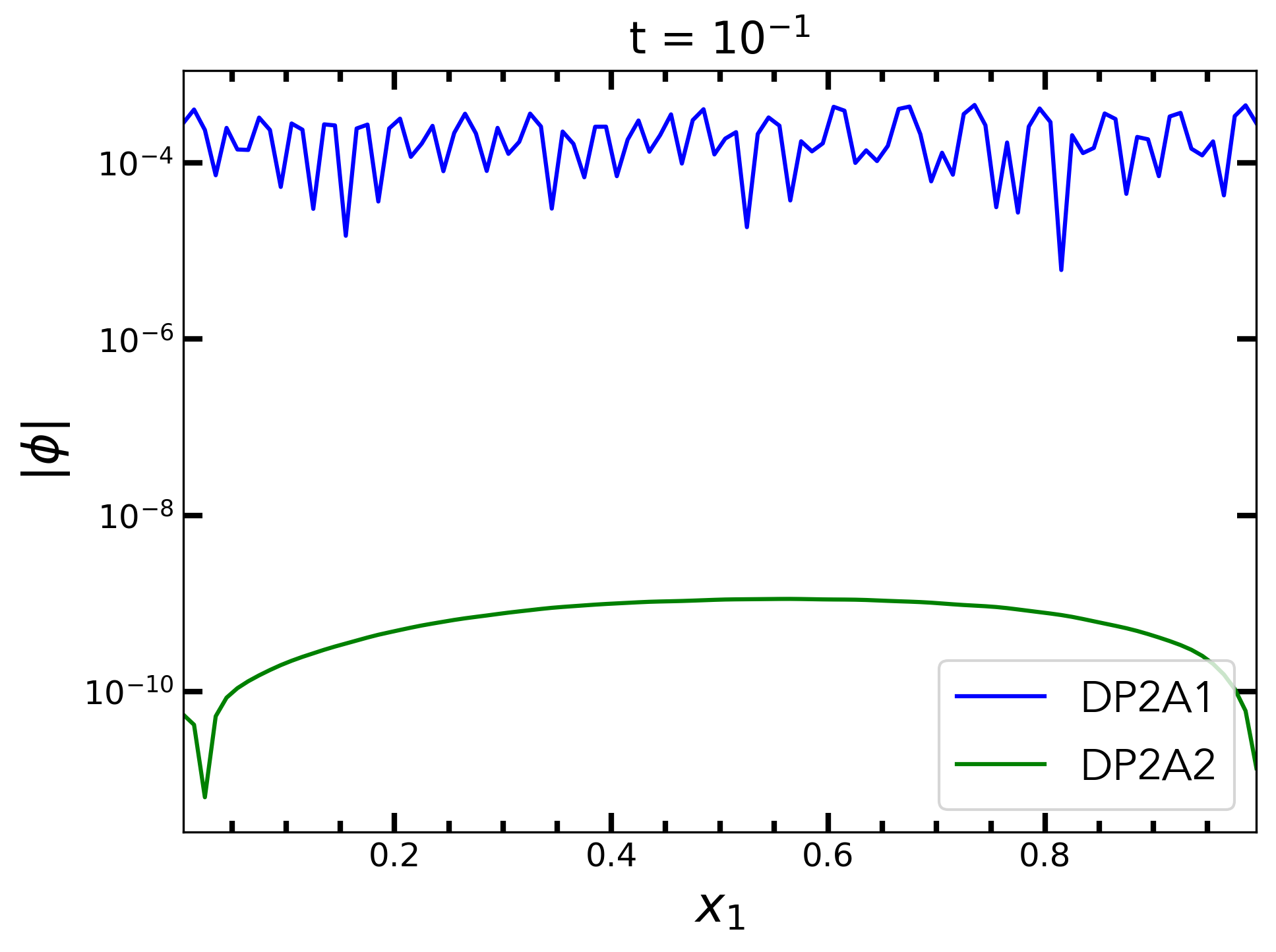}
  \caption{Case 2, under-resolved mesh $N = 10^{2}$. 
  with Penalised type-A IMEX Schemes. 
  $\lambda = 10^{-4}$. Left: $x_1\to |\rho(t, x_1)-1|$, Center:
    $x_1\to |\dvg u(t, x_1)|$, Right: $x_1\to |\phi(t, x_1)|$, at time $t
    = 0.1$, $K = 16$.}  
  \label{fig:Small_pertur_QN_ur_0p1_nwp}
\end{figure}
\begin{figure}[htbp]
  \centering
  \includegraphics[height=0.162\textheight]{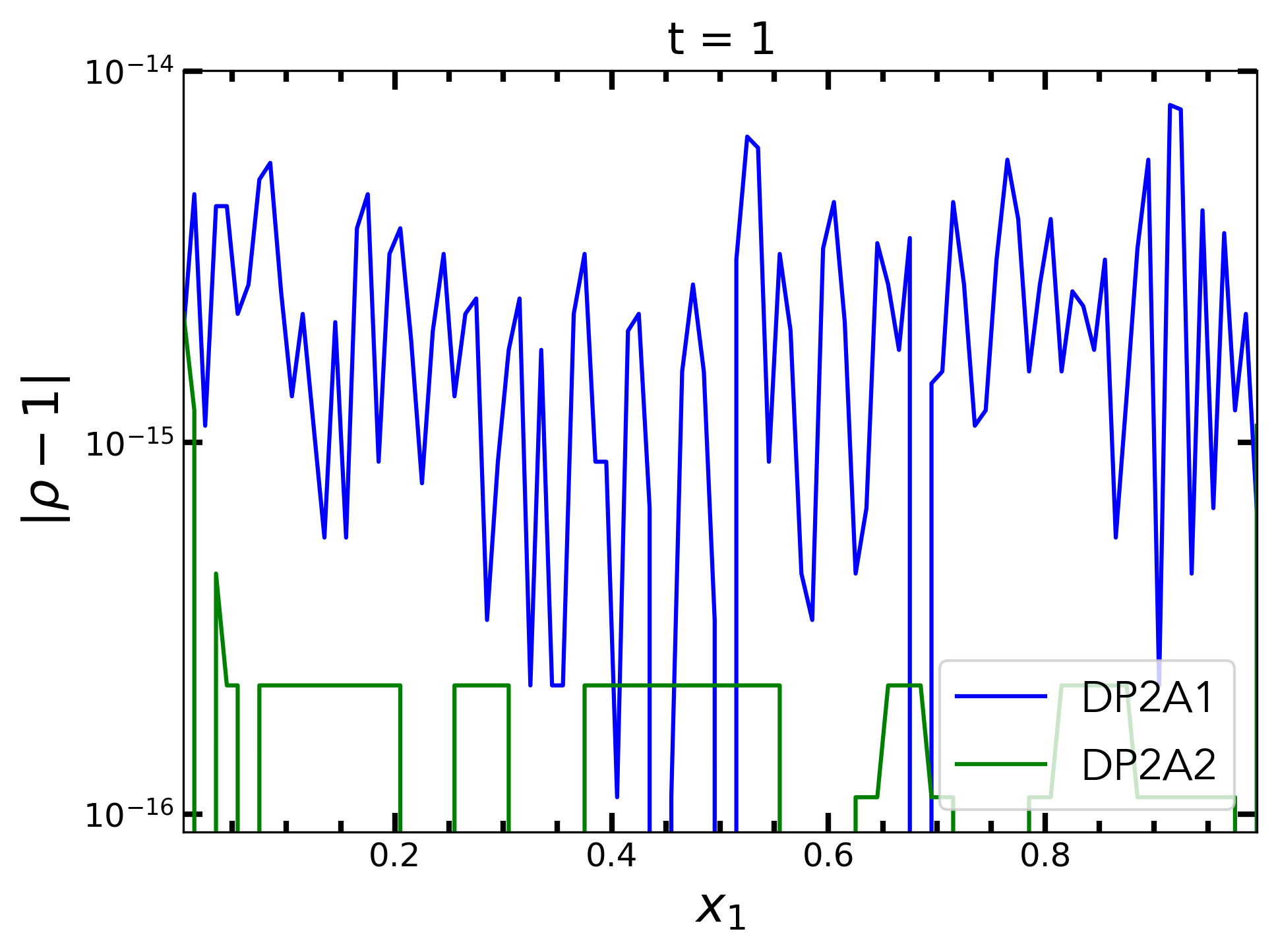}
  \includegraphics[height=0.162\textheight]{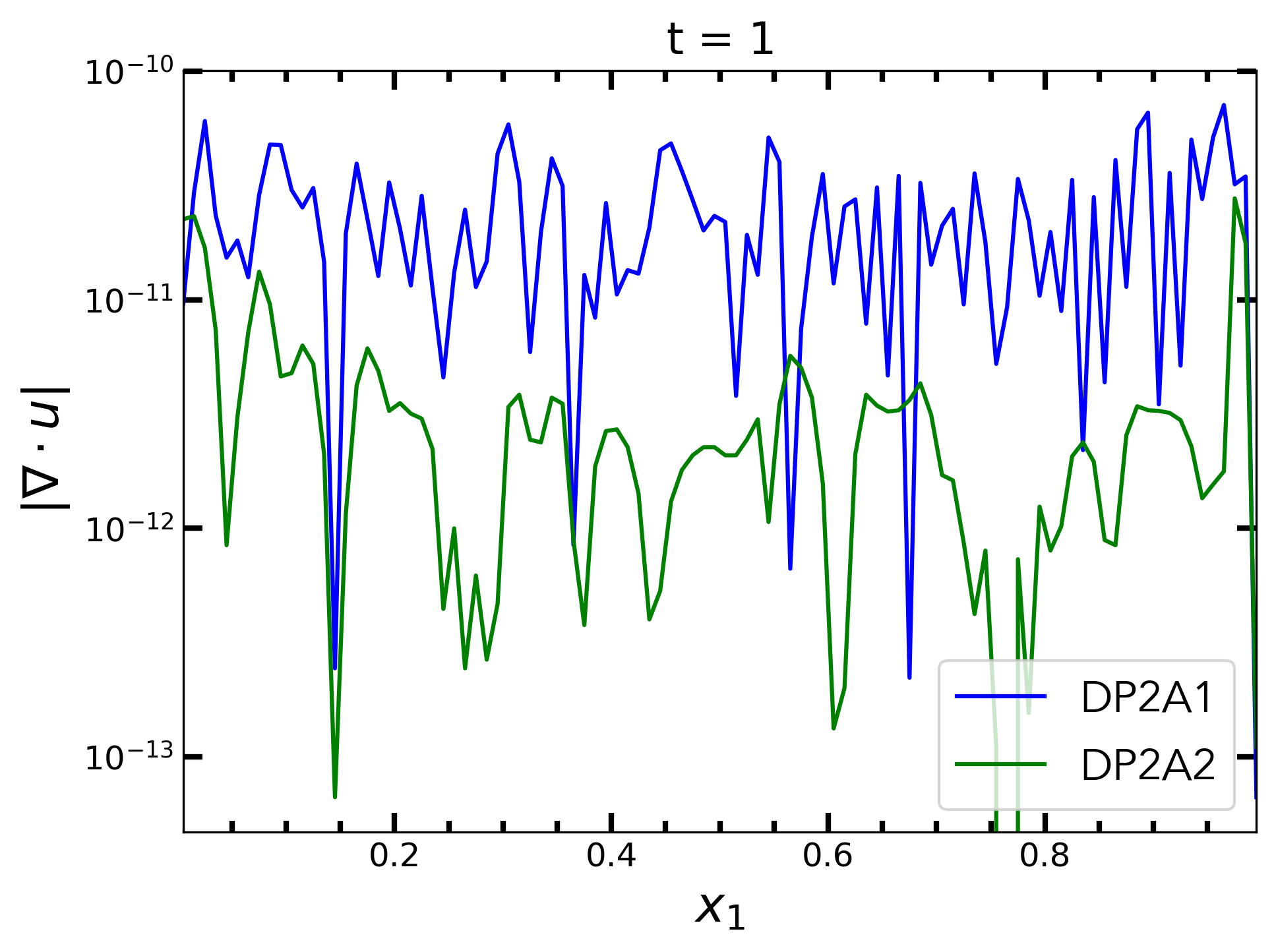}
  \includegraphics[height=0.162\textheight]{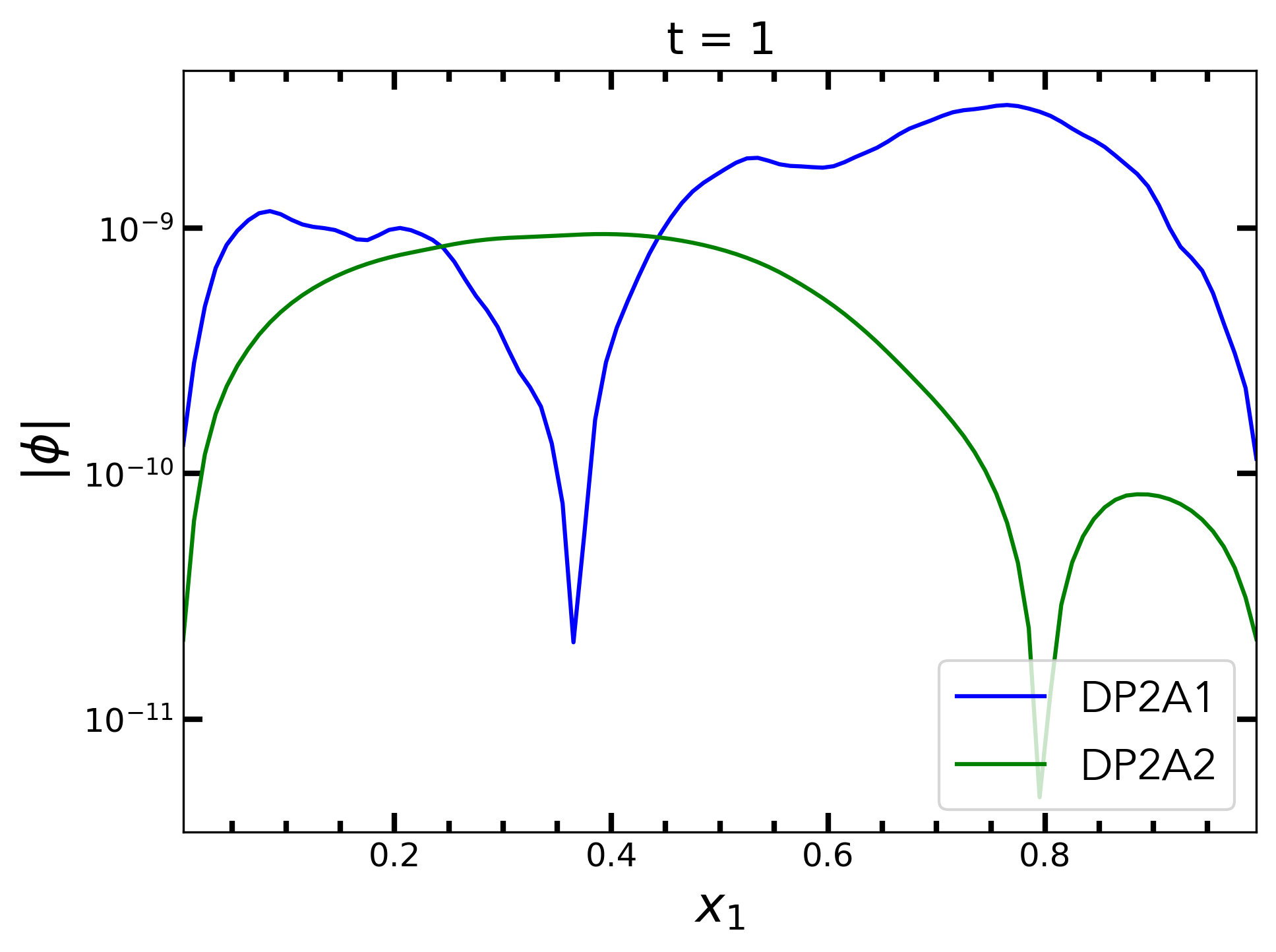}
  \caption{Case 2, under-resolved mesh $N = 10^{2}$. 
  with Penalised type-A IMEX Schemes. 
  $\lambda = 10^{-4}$. Left: $x_1\to |\rho(t, x_1)-1|$, Center:
    $x_1\to |\dvg u(t, x_1)|$, Right: $x_1\to |\phi(t, x_1)|$, at time $t
    = 1.0$, $K = 16$.}  
  \label{fig:Small_pertur_QN_ur_1p0_nwp}
\end{figure}

Therefore, this corroborates the conclusion of the analysis presented via Proposition~\ref{prop1} that even if the simulation was subjected to a data which is not well-prepared the type-A penalised IMEX scheme for the EP system could bring the solution closer to the asymptotic solution.

Figure~\ref{fig:Small_pertur_QN_ur_0p025_nwp} shows the plots of the same quantities as that of Figure~\ref{fig:Small_pertur_QN_ur_0p1_nwp} but for the type-CK (ARS) penalised-IMEX scheme at time $t = 0.025$. We can clearly see that the type-CK scheme is not only not able to maintain a solution close to the asymptotic regime i.e $\rho - 1 = \mathcal{O}(\lambda^2)$, $\dvg u = \mathcal{O}(\lambda^2)$ and $\phi = \mathcal{O}(1)$, but it does far worse. The solution starts blowing up in a very very small time. 
\begin{figure}[htbp]
  \centering
  \includegraphics[height=0.162\textheight]{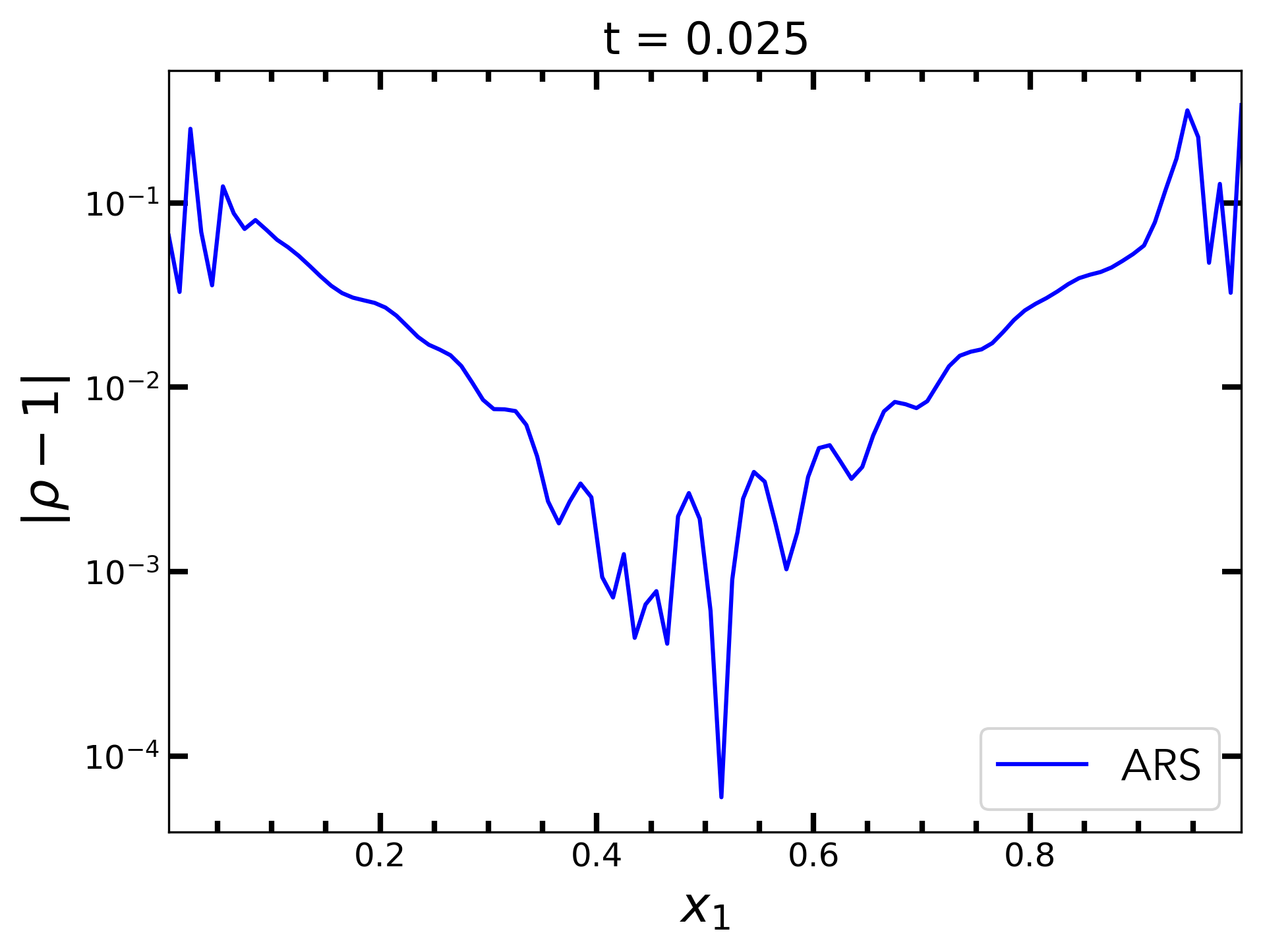}
  \includegraphics[height=0.162\textheight]{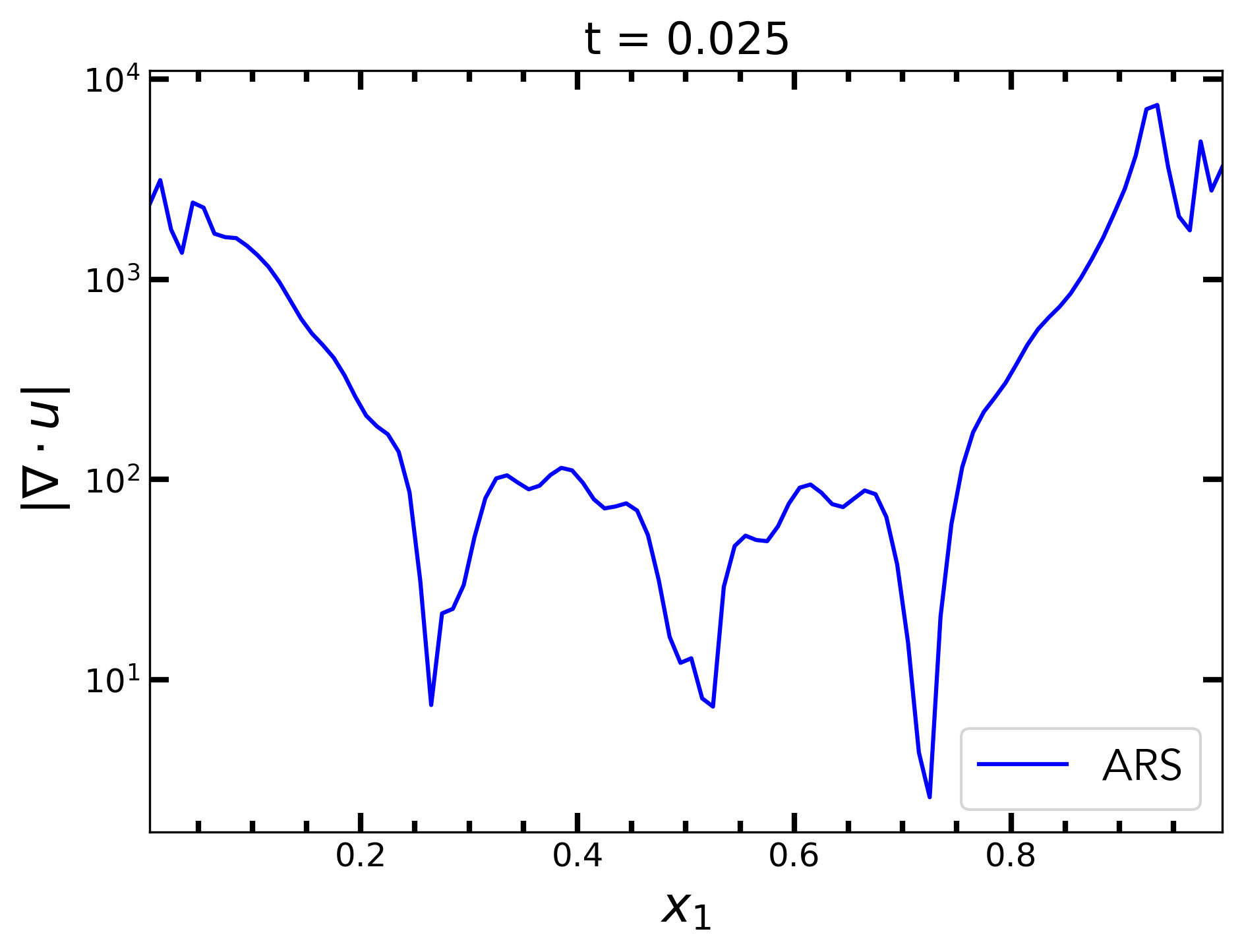}
  \includegraphics[height=0.162\textheight]{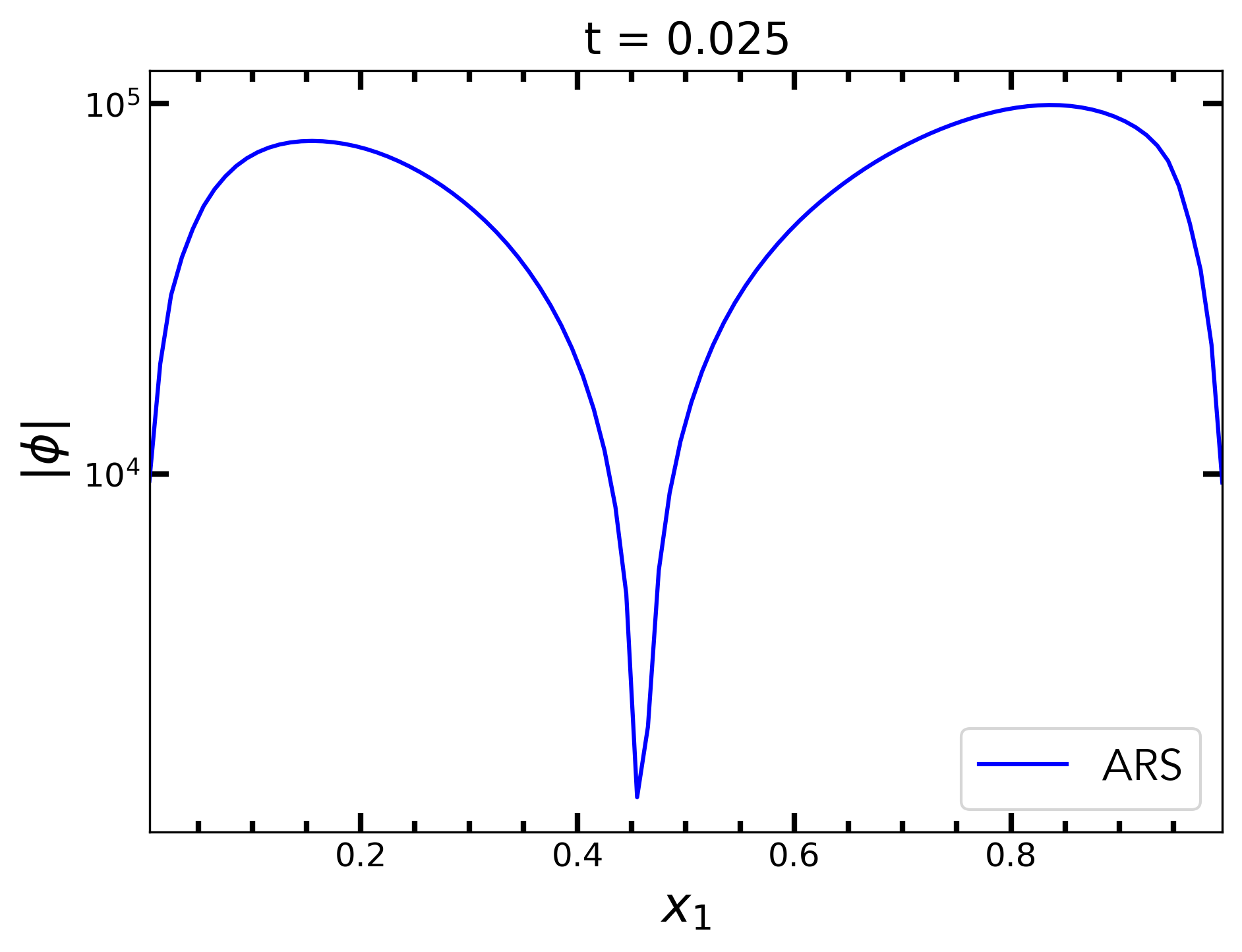}
  \caption{Case 2, under-resolved mesh $N = 10^{2}$. 
  with Penalised CK-IMEX Schemes. $\veps = 10^{-4}$,
    $N = 10^{-2}$ (Under resolved mesh). Left: $x_1\to |\rho(t, x_1)-1|$, Center:
    $x_1\to |\dvg u(t, x_1)|$, Right: $x_1\to |\phi(t, x_1)|$, at time $t
    = 0.025$, $K = 16$.}  
  \label{fig:Small_pertur_QN_ur_0p025_nwp}
\end{figure}

\subsection{Slight Perturbation of a Maxwellian}
Both the classical higher order additive IMEX-RK-DAE scheme presented in \cite{ACS24} and the penalised-IMEX schemes; see Definition~\ref{def:fully_disc_schm}, are developed  based on the same additive IMEX platform \eqref{eq:imex_Yi}-\eqref{eq:imex_yn+1}. In \cite{ACS24} it was shown that the classical higher order schemes are plagued with stability issues when subjected to a coarse mesh without the time step restriction of $\Delta t < \lambda$. This test case is aimed at showing the superiority of the new high order penalised-IMEX schemes over the high order classical schemes. Even though developed on the same additive IMEX framework, the schemes developed in this paper owe some robustness due to the stabilising penalisation. To show that the choice of the Butcher tableau is just one aspect to get a good scheme but it has to be coupled with a prudent splitting (or choice of implicit and explicit terms), we choose the same Butcher tableau i.e. DP2A2 for both the classical scheme (CL-DP2A2) and the penalised-IMEX scheme (AP-DP2A2).
 
Here we consider a non well-prepared initial data. The test case is from \cite{Neg13}, wherein a small perturbation of amplitude $\delta = 10^{-2}$ is added to a Maxwellian. The initial data is stated as:
\begin{equation}
    \rho(0, x_1) = 1.0 + \delta \sin (\kappa \pi x_1), \;\; u(0, x_1)=0, 
\end{equation}
where the frequency is $\kappa = 2220$. The value of $\kappa$ is chosen such that $\kappa \sim \lambda^{-1}$ to ensure that the wavelength of the density perturbation is of the same order as the Debye length. This is a slight perturbation of the steady state equilibrium, so  {initially the plasma is at rest. 
\begin{figure}[htbp]
  \centering
  \includegraphics[height=0.23\textheight]{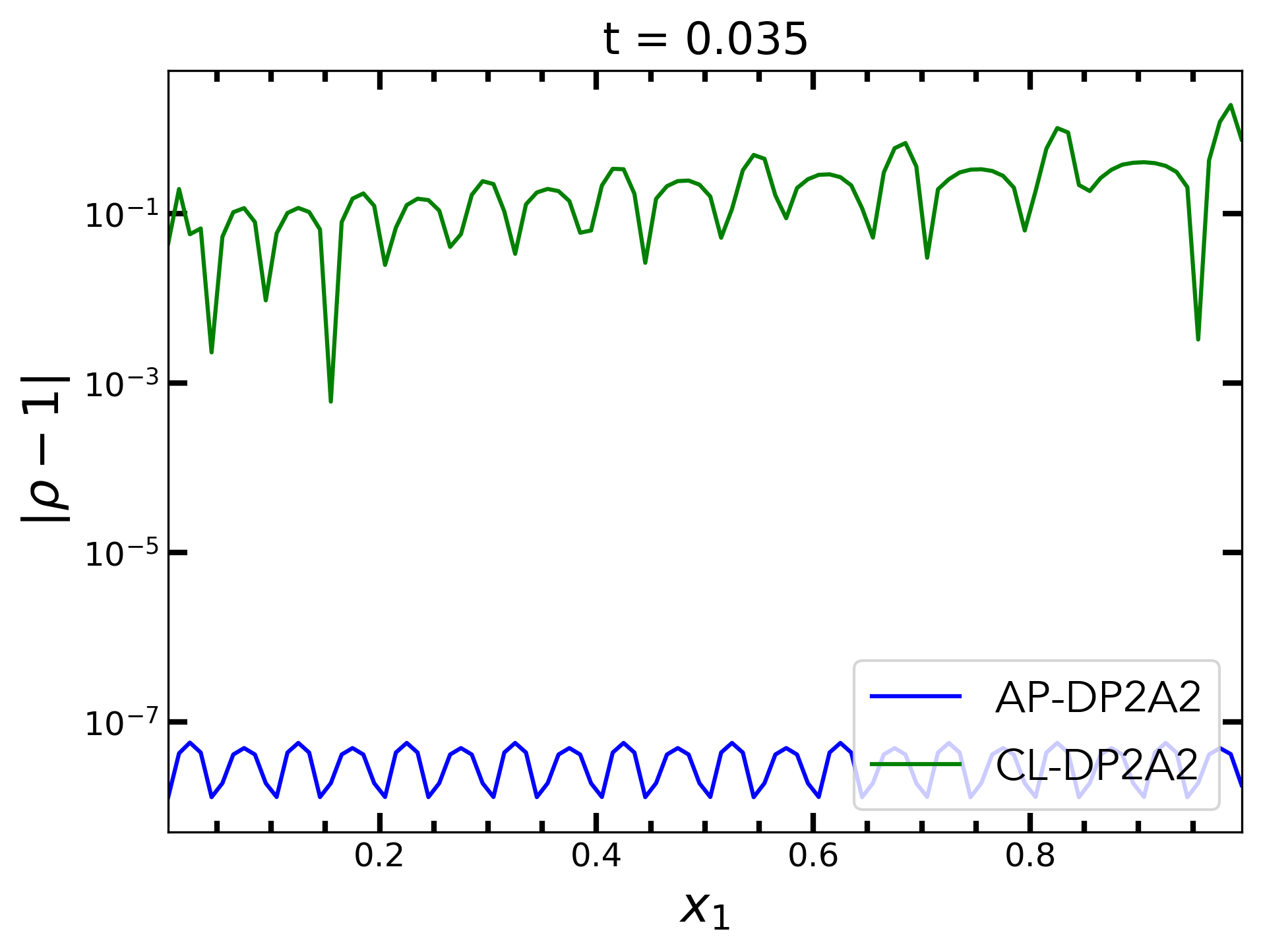}
  \
  \includegraphics[height=0.23\textheight]{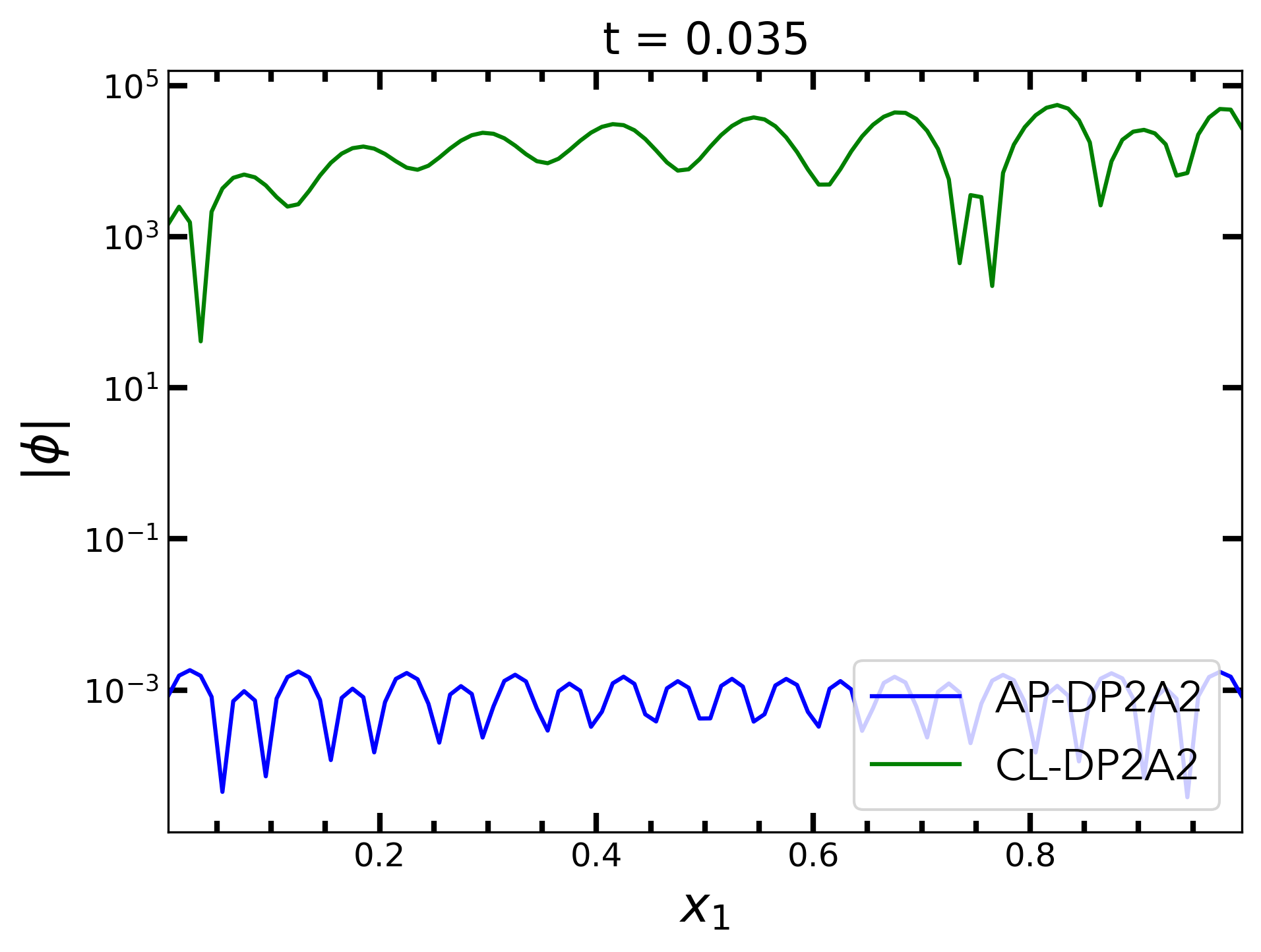}
  \caption{Slight perturbation of a Maxwellian with Penalised type-A IMEX and Classical schemes. $\lambda = 10^{-4}$,
    $N = 10^2$ (Unresolved mesh). Left: $x_1\to |\rho(t, x_1)-1|$ and
    Right: $x_1\to |\phi(t, x_1)|$, at time = 0.035}  
  \label{fig:Small_max_pert_unres}
\end{figure}
For all the plots in this subsection the $y$-axis is in log scale. Figure~\ref{fig:Small_max_pert_unres} shows the plots of the density perturbation ($|\rho - 1|$) and absolute value of electric potential computed with additive IMEX schemes, with penalisation (AP-DP2A2) and without penalisation (CL-DP2A2). The spatial mesh parameters do not resolve $\lambda$. For the classical scheme (CL-DP2A2) the time mesh steps are chosen to be $\Dlt = 10^{-4} = \lambda$ (required for stability), which resolves $\lambda$. Whereas for the penalised-IMEX scheme (AP-DP2A2) the time step are uniformly chosen to be $\Dlt = 0.5 \Delta x_1$, not resolving $\lambda$. Clearly, even after getting extra stability restrictions the classical scheme doesn't perform well and the potential starts blowing up. On the other side, in spite of generous stability restriction the new Penalised-IMEX schemes are able to maintain the equilibrium state quite accurately, despite the coarse spatial mesh. Figure~\ref{fig:Small_max_pert_unres_t0p1} shows the same quantities as Figure~\ref{fig:Small_max_pert_unres} but just for the penalised-IMEX scheme at a much later time $t = 0.1$. We can clearly see that the numerical solution is very close to the equilibrium state.
\begin{figure}[htbp]
  \centering
  \includegraphics[height=0.23\textheight]{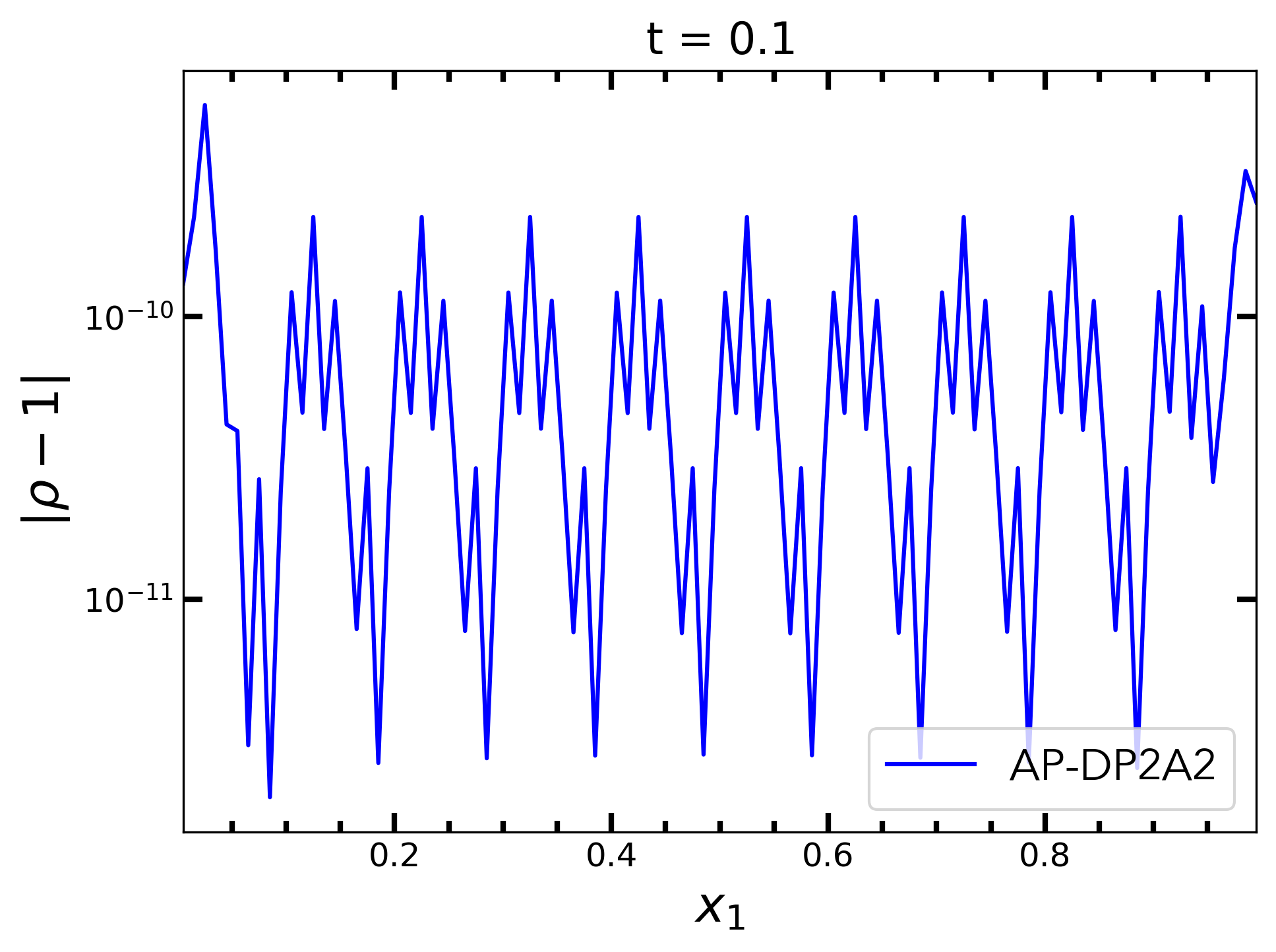}
  \
  \includegraphics[height=0.23\textheight]{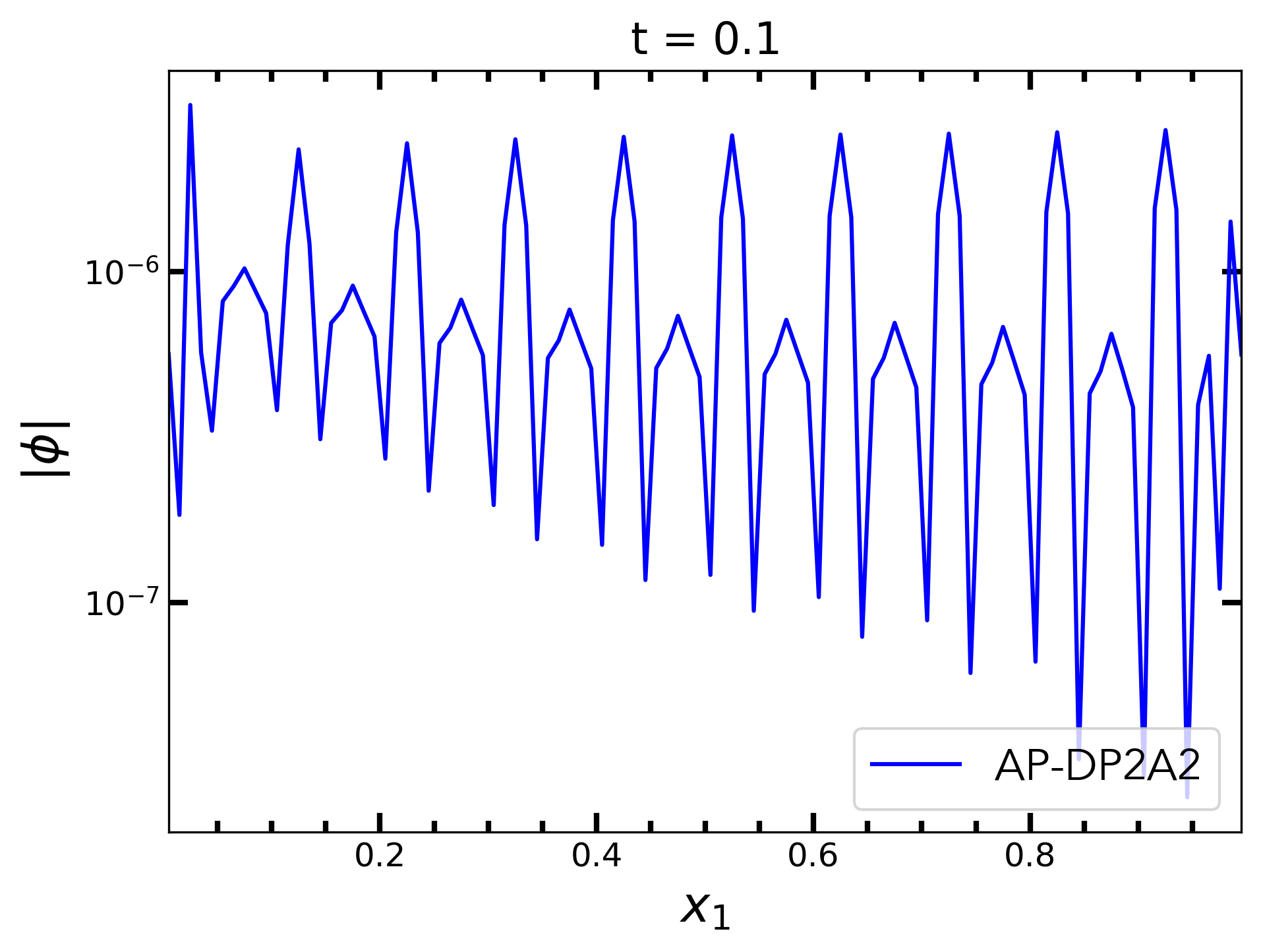}
  \caption{Slight perturbation of a Maxwellian with Penalised-AP scheme. $\veps = 10^{-4}$,
    $N = 10^2$ (Unresolved mesh).Left: $x_1\to |\rho(t, x_1)-1|$ and
    Right: $x_1\to |\phi(t, x_1)|$, at time = 0.1}  
  \label{fig:Small_max_pert_unres_t0p1}
\end{figure}

\subsection{Asymptotic Order of Convergence (AOC)}
\label{sec:ep_numer_case_stud_P2}
In this subsection, we want to illustrate the order of convergence of the asymptotic scheme. 
To do so, we set the following experiment to numerically validate the asymptotic order  convergence of the numerical solution to the
solution of the asymptotic model.  

In order to do this, we compute the 
asymptotic order of convergence, as is defined in \cite{AS20}, as the 
experimental rate of convergence of a numerical scheme with respect to 
the asymptotic limit solution. Therefore the exact solution, used as a 
reference, is the solution of the incompressible limit
\eqref{eq:ep_mass_lim}-\eqref{eq:ep_poi_lim_}.

 Let $U_{\lambda}(t^n)$ be the solution of the EP system \eqref{eq:ep_nd_mass}-\eqref{eq:ep_nd_poi} and $U_{(0)}(t^n)$ the solution of the incompressible
limit system \eqref{eq:ep_mass_lim}-\eqref{eq:ep_poi_lim_}, at time $t^n$. Let $U_{\lambda}^n$ be the numerical solution obtained using any second order accurate Penalised IMEX scheme for the EP system \eqref{eq:ep_nd_mass}-\eqref{eq:ep_nd_poi} at the same
time. Then we expect the following estimates:
\begin{align}
U_{\lambda}^{n} - U_{(0)}^n &= U_{\lambda}^{n} -U_{\lambda}(t^n)  +  U_{\lambda}(t^n)  - U_{(0)}(t^n) +  U_{(0)}(t^n) - U_{(0)}^n \\
  & = \mathcal{O}(\Dlt^2) + \mathcal{O}(\lambda) + \mathcal{O}(\Dlt^2).
\end{align}
Therefore from the above calculations, we expect when  $\lambda \to 0$
\begin{equation}
 U_{\lambda}^{n} - U_{(0)}^n  = \mathcal{O}(\Dlt^2). 
\end{equation}
To this end we take the initial conditions as considered in
Section~\ref{sec:ep_numer_case_stud_P1} with a modification in the
simulation domain and the value of the perturbation parameter. Particularly, we
expand the plasma domain so that the initial perturbation can travel
within the domain before encountering the boundary conditions at the
domain walls. Consequently, the spatial domain is set to be the
interval $x\in [0, 20]$, and the perturbation parameter is
increased to $\delta = 10^{-2}$ and the parameter $K$ is set to be $1$.

Here we have considered three different values of $\lambda \in \{10^{-4},
10^{-5}, 10^{-6} \}$, all in the asymptotic range. As defined in \cite{AS20}, the experimental order of convergence computed using incompressible data  $(\rho_{(0)} ,u_{1,(0)}, \phi_{(0)})$ as the reference solution is termed as the asymptotic
order of convergence (AOC).
\begin{table}
  \begin{tabular}[t] {
|p{1.3cm}|p{2.2cm}|p{0.8cm}||p{2.2cm}|p{0.8cm}||p{2.2cm}|p{0.8cm}|} 
    \hline
    & $\lambda = 10^{-4}$ & & $\lambda = 10^{-5}$  & & $\lambda = 10^{-6}$
    & \\
    \hline
    $N$& $L^2$ Error in $\phi$ & AOC & $L^2$ Error in
$\phi$   & AOC & $L^2$ Error in $\phi$ & AOC  \\
    \hline
    320  & 3.2851E-03 & & 3.2896E-03 & & 3.2886E-03 & \\
    640 & 8.1925E-04 & 2.00 & 8.2002E-04 & 2.00 & 8.2648E-04 &1.99\\
    1280 & 2.4011E-04 & 1.77 & 2.4018E-05 & 1.77  & 2.3955E-05 & 1.78\\
    2560  & 6.2858E-05 & 1.93 & 6.2700E-06 & 1.93 & 8.0036E-06 & 1.58\\
    \hline
  \end{tabular}
  \caption{ AOC for Penalised-IMEX (DP2A2) scheme for $\lambda
      = 10^{-4}, 10^{-5}, 10^{-6}$.}
  \label{table:ep_AOC}
\end{table}
Table~\ref{table:ep_AOC} shows the uniform second order convergence of the penalised-IMEX scheme in the QN regime. 

\subsection{Perturbation of QN Plasma in 2D}
\label{subsec:nr_2d}
This test case is a two-dimensional version of the test considered in
Section~\ref{sec:ep_numer_case_stud_P1}. Indeed, we consider the following initial condition where the perturbation is considered on a (divergence free) velocity field 
\begin{equation}
  \begin{aligned}
    \rho (0, x_1, x_2) &  = 1., \\
    \ u_1(0, x_1, x_2) & = 1. + \sin (K \pi (x_1 - x_2)) + \lambda \sin (K \pi (x_1 + x_2)), \\
    \ u_2(0, x_1, x_2) &  = 1. + \sin (K \pi (x_1 - x_2)) + \lambda \cos (K \pi (x_1 + x_2)), \\
    \phi (0, x_1, x_2) &  = 0, 
  \end{aligned}
\end{equation}
where $(x_1, x_2)\in [0, 1]^2$ and $K=16$. 
$$
||\rho - 1||_{L^2} = \mathcal{O}(\lambda^2) \qquad ||\dvg u||_{L^2} = \mathcal{O}(\lambda).
$$
  Note that the perturbation chosen here pertains to a set of
  non-well-prepared initial data, cf.\ Definition \ref{def:wp_data}.
  The density is well-prepared, whereas the velocity even though plagued with
  a small perturbation ($\mathcal{O}(\lambda)$), is still not well-prepared.  
Throughout this test, the Penalised-type A IMEX scheme is   subjected to the advective CFL condition~\eqref{eq:ep_CFL}. The CFL
  number is set to be $0.45$. Periodic boundary conditions are imposed 
  for all the components and a mean-free condition on $\phi$ is added to 
  ensure its uniqueness. The value of the isentropic parameter
  $\gamma$ for the pressure is taken to be $2$: $p(\rho) = \rho^2$.
\begin{figure}[htbp]
  \centering
  \includegraphics[height=0.24\textheight]{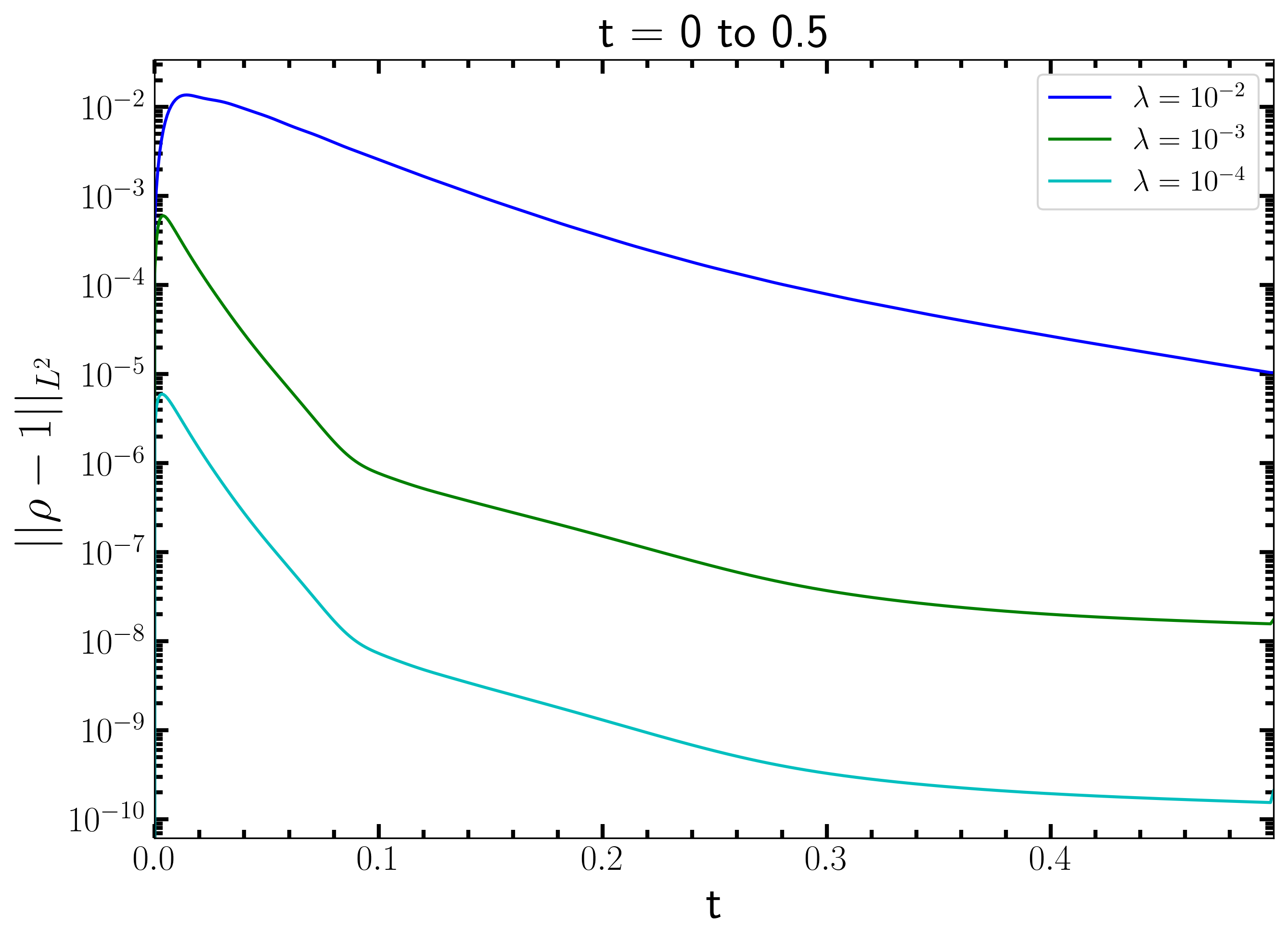}
  \includegraphics[height=0.24\textheight]{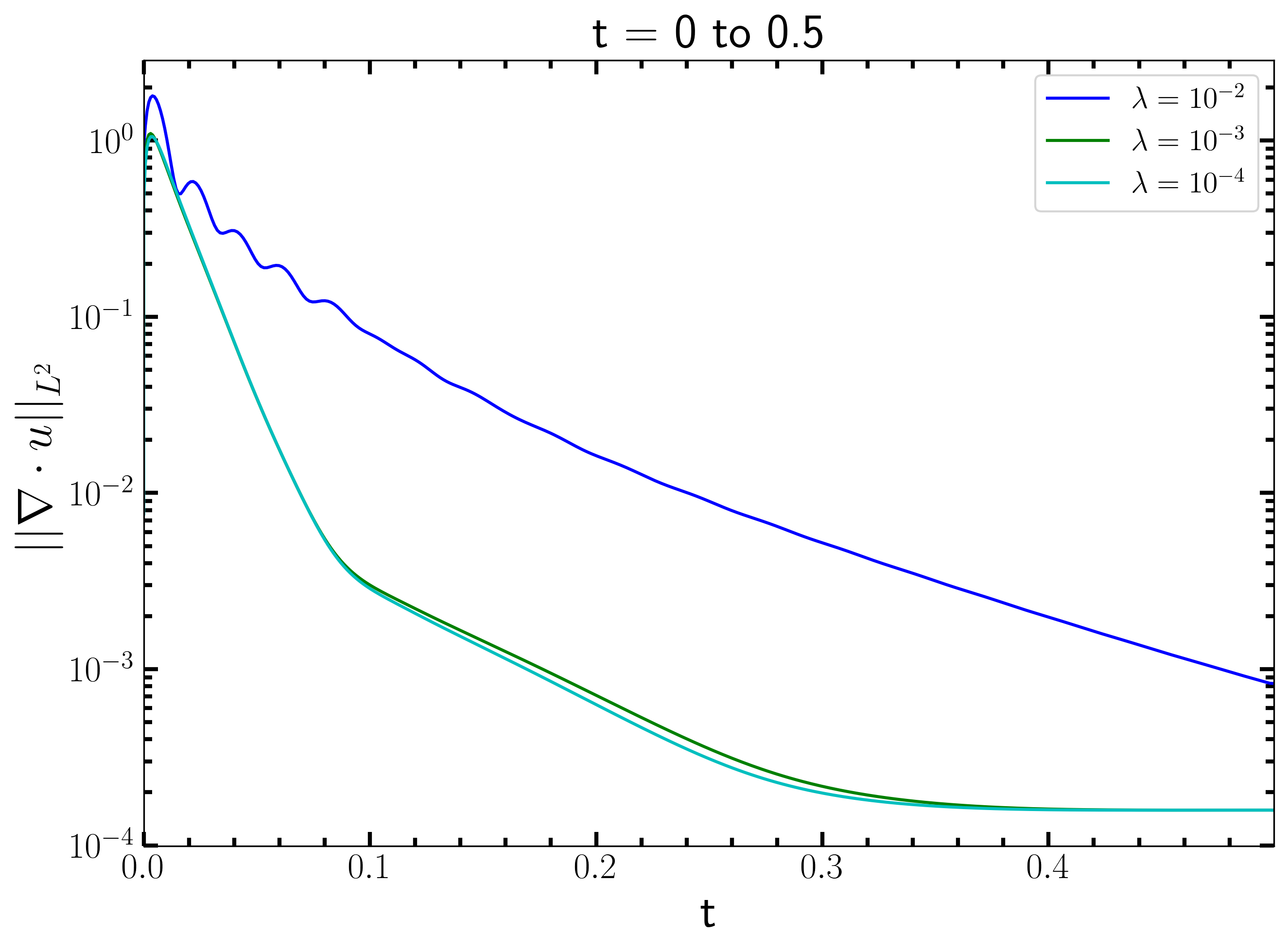}
  \caption{2D test: Small perturbation of a quasineutral state with Penalised-IMEX (DP2A2) scheme. Left: $t \to ||\rho(t, x_1, x_2) - 1||_{L^2}$ and Right: $t \to ||\nabla \cdot u (t, x_1, x_2)||_{L^2}$, from time $t = 0$ to $t = 0.5$, for $\lambda =  10^{-2}, 10^{-3}$ and $10^{-4}$.}
  \label{fig:2d_L2_vs_t}
\end{figure}

Figure~\ref{fig:2d_L2_vs_t} shows the time evolution of the asymptotic parameters under study $||\rho - 1||_{L^2}$ and $|| \dvg u||_{L^2}$ in log scale, for the time domain $t \in [0.0.5]$ for $\lambda$ in the range $\{10^{-2}, 10^{-3}, 10^{-4} \}$. For each of the values of $\lambda$ we see that the penalised-IMEX scheme is able to maintain the initial proximity of the solution to the QN limit. This shows that the scheme is asymptotically consistent with the QN limit system.

\section{Conclusions} 
In this paper, we developed and analyzed a class of high-order penalized Implicit-Explicit (IMEX) schemes for the Euler-Poisson (EP) system in the quasineutral limit. These schemes were designed to be uniformly stable with respect to the Debye length and to degenerate into accurate discretizations of the quasineutral system as the Debye length approaches zero. This characteristic allows the schemes to efficiently handle the coexistence of quasineutral and non-quasineutral regions, which is a possilbe scenario in practical plasma applications.

We addressed one of the major limitations of previous studies, which assumed well-prepared initial conditions. In contrast, the schemes proposed in this work are capable of handling more general initial conditions without losing stability or consistency with the asymptotic limit. This broadens the applicability of the numerical methods to more realistic plasma scenarios where such assumptions cannot always be guaranteed.

Our analysis and numerical tests demonstrated that the proposed schemes possess the desired Asymptotic-Preserving (AP) property, ensuring that they remain accurate and stable even in the stiff quasineutral regime. We also explored the performance of different IMEX schemes, showing that certain types (such as type-A IMEX Runge-Kutta methods) are more suitable for this class of problems compared to others (such as type-CK schemes), which may fail under certain conditions.

In conclusion, the penalized IMEX schemes introduced in this paper offer a robust and computationally efficient solution for simulating plasma phenomena in the quasineutral limit, making them a valuable tool. Future work could extend the schemes here presented to more general models and to higher-dimensional problems.
 
\bibliographystyle{acm}
\bibliography{ref}
\end{document}